%% file: DM.tex
\documentclass[reqno,11pt]{amsart}
\usepackage{amssymb}
\usepackage{mathrsfs}
\usepackage{txfonts}
\usepackage{amsfonts}
\usepackage{amstext}
\usepackage{amssymb}
\usepackage{amsmath}
\usepackage{graphicx}
\usepackage{hyperref}
\usepackage{url}

\hypersetup{
        colorlinks   = true,
        citecolor    = blue,
        linkcolor    = blue,
        urlcolor     = blue
}
\allowdisplaybreaks
\newtheorem{theorem}{Theorem}[section]
\newtheorem{proposition}[theorem]{Proposition}
\newtheorem{lemma}[theorem]{Lemma}

\newtheorem{remark}[theorem]{Remark}
\renewcommand{\theequation}{\thesection.\arabic{equation}}

\numberwithin{equation}{section}
\newcounter{counterConstant}

\newenvironment{definition}[1][Definition]{\bigskip\noindent\textbf{#1. }\it}{\rm\bigskip}
\newenvironment{definitions}[1][Definition]{\bigskip\noindent\textbf{#1. }\it}{\rm}
\input{tcilatex}
\input{tcirm}

\begin{document}
\title[Davies' method]{The Davies method revisited for heat kernel upper
bounds of regular Dirichlet forms on metric measure spaces}
\author[Hu]{Jiaxin Hu}
\address{Department of Mathematical Sciences, Tsinghua University, Beijing
100084, China.}
\email{hujiaxin@mail.tsinghua.edu.cn}
\author[Li]{Xuliang Li}
\address{Department of Mathematical Sciences, Tsinghua University, Beijing
100084, China.}
\email{lixuliang12@mails.tsinghua.edu.cn}
\thanks{\noindent JH was supported by NSFC No.11371217, SRFDP
No.20130002110003.}
\date{April 2017}

\begin{abstract}
We apply the Davies method to prove that for any regular Dirichlet form on a
metric measure space, an off-diagonal stable-like upper bound of the heat
kernel is equivalent to the conjunction of the on-diagonal upper bound, a
cutoff inequality on any two concentric balls, and the jump kernel upper
bound, for any walk dimension. If in addition the jump kernel vanishes, that
is, if the Dirichlet form is strongly local, we obtain sub-Gaussian upper
bound. This gives a unified approach to obtaining heat kernel upper bounds
for both the non-local and the local Dirichlet forms.
\end{abstract}

\maketitle
\tableofcontents

\section{Introduction}

We are concerned with heat kernel upper bounds for both nonlocal and local
Dirichlet forms on metric measure spaces.

Let $(M,d)$ be a locally compact separable metric space and $\mu $ be a
Radon measure on $M$ with full support, and the triple $(M,d,\mu )$ is
called a \emph{metric measure space}. Let $\left( \mathcal{E},\mathcal{F}%
\right) $ be a regular \emph{Dirichlet form} in $L^{2}\left( M,\mu \right) $%
, and $\mathcal{L}$ be its generator (non-positive definite self-adjoint).
Let 
\begin{equation*}
\left\{ P_{t}=e^{t\mathcal{L}}\right\} _{t\geq 0}
\end{equation*}%
be the associated heat semigroup. Recall that the form $(\mathcal{E},%
\mathcal{F})$ is \emph{conservative} if $P_{t}1=1$ holds for all $t>0$.

Let $\Omega $ be a non-empty open set on $M$, let $\mathcal{F}(\Omega )$ be
the closure of $\mathcal{F}\cap C_{0}(\Omega )$ in the norm of $\mathcal{F}$%
, where $C_{0}(\Omega )$ is the space of all continuous functions with
compact supports in $\Omega $. It is known that if $(\mathcal{E},\mathcal{F}%
) $ is regular, then $(\mathcal{E},\mathcal{F}(\Omega ))$ is a regular
Dirichlet form in $L^{2}(\Omega ,\mu )$ (cf. \cite[Lemma 1.4.2 (ii) p.29]%
{fukushima2010dirichlet}). We denote by $\mathcal{L}^{\Omega }$ the
generator of $(\mathcal{E},\mathcal{F}(\Omega ))$ and by $\{P_{t}^{\Omega
}\} $ the associated semigroup.

A family $\left\{ p_{t}\right\} _{t>0}$ of non-negative $\mu \times \mu $%
-measurable functions on $M\times M$ is called the \emph{heat kernel} of $%
\left( \mathcal{E},\mathcal{F}\right) $ if for any $f\in L^{2}(M,\mu )$ and $%
t>0$, 
\begin{equation*}
P_{t}f(x)=\int_{M}p_{t}(x,y)f(y)d\mu (y)
\end{equation*}%
for $\mu $-almost all $x\in M$.

Typically, there are two distinct types of heat kernel estimates on \emph{%
unbounded} metric spaces, depending on whether the form $\left( \mathcal{E},%
\mathcal{F}\right) $ is local or not. Indeed, assume that the heat kernel
exists and satisfies the following estimate 
\begin{equation}
p_{t}\left( x,y\right) \asymp \frac{C}{t^{\alpha /\beta }}\Phi \left( \frac{%
d(x,y)}{ct^{1/\beta }}\right)  \label{90}
\end{equation}%
with some function $\Phi $ and two positive parameters $\alpha ,\beta $,
where the sign $\asymp $ means that both $\leq $ and $\geq $ are true but
with different values of $C,c.$ Then either $\Phi \left( s\right) =\exp
\left( -s^{\frac{\beta }{\beta -1}}\right) $ (thus $\left( \mathcal{E},%
\mathcal{F}\right) $ is local), or $\Phi \left( s\right) =\left( 1+s\right)
^{-\left( \alpha +\beta \right) }$ (thus $\left( \mathcal{E},\mathcal{F}%
\right) $ is non-local), see \cite{grku2008dic}. For the local case, the
heat kernel $p_{t}(x,y)$ admits the following \emph{Gaussian }($\beta =2$)-
or \emph{Sub-Gaussian }($\beta >2$) estimate: 
\begin{equation}
p_{t}(x,y)\asymp \frac{C}{t^{\alpha /\beta }}\exp \left( -c\left( \frac{%
d(x,y)}{t^{1/\beta }}\right) ^{\beta /(\beta -1)}\right) ,
\label{eq:sub_G_es}
\end{equation}%
where $\alpha >0$ is the Hausdorff dimension and $\beta \geq 2$ is termed
the walk dimension, see for example \cite{barlow1998diffusions, bbkt10,
barlow1999brownian, barlow1988Brownian, hambly1999transition, kig09}. Some
equivalence conditions are stated in \cite%
{barlow2012equivalence,grigor2014heatkernels, grigor2014CE, grigor2012two}.
On the other hand, for the non-local case, the heat kernel $p_{t}(x,y)$
admits the \emph{stable-like} estimates: 
\begin{equation}
p_{t}(x,y)\asymp \frac{1}{t^{\alpha /\beta }}\left( 1+\frac{d(x,y)}{%
t^{1/\beta }}\right) ^{-\left( \alpha +\beta \right) }
\label{eq:HK_nonlocal}
\end{equation}%
where $\alpha >0$ and $\beta >0$, see for example, \cite{bagrku2009hkub,
bale2002hk, chenku2003hkdsets, chenku2008hkjump} for $0<\beta <2$, and \cite%
{ckw16, ghh16HE, ghh16LE, grhulau2014nonlocal} for any $\beta >0$. Note that
estimate (\ref{eq:HK_nonlocal}) can also be obtained by using the
subordination technique, see for example \cite%
{grigor2003heatkernel,huz2009br, kumagai2003remark,stos2000stable}. It was
shown in \cite{grku2008dic} that estimates (\ref{eq:sub_G_es}) and (\ref%
{eq:HK_nonlocal}) exhaust all possible two-sided estimates of heat kernels
upon assuming (\ref{90}).

Recently, Murugan and Saloff-Coste extend the Davies method developed in 
\cite{carlen1987upper,Dav87} and obtain heat kernel upper bounds, for local
Dirichlet forms on metric spaces in \cite{murugan2015davies} and for
non-local Dirichlet forms on infinite graphs in \cite{murugan2015heat},
where a cutoff inequality introduced in \cite{andres2013energy} plays an
important role.

The purpose of this paper is twofold:

\begin{itemize}
\item[$(1)$] to extend the result in \cite{murugan2015heat} to the metric
measure space;

\item[$(2)$] to unify the Davies method for both local and nonlocal
Dirichlet forms.
\end{itemize}

More precisely, we give some equivalence characterizations of heat kernel
upper bounds both in (\ref{eq:HK_nonlocal}) for any $\beta >0$ and in (\ref%
{eq:sub_G_es}) for any $\beta >1$, see Theorem \ref{T2} below, by applying
the Davies method in a unified way. These characterization are stable under
bounded perturbation of Dirichlet forms. We mention that one of our starting
point here is from the \emph{cutoff inequality on balls }to be stated below,
labelled by condition (\ref{eq:CSA}), which is subtly distinct from the
similar conditions in previous papers \cite{andres2013energy}, \cite%
{grigor2014CE}, \cite{murugan2015davies, murugan2015heat}, \cite{ckw16}, 
\cite{ghh16HE}. Also the metric space considered in this paper is allowed to
be bounded or unbounded, unlike the most previous ones in which the metric
space is always assumed to be unbounded.

Let us return to the general setup of a metric measure space $(M,d,\mu )$
equipped with a regular Dirichlet form $(\mathcal{E},\mathcal{F})$. Assume
that $\mathcal{E}$ admits the following decomposition without \emph{killing
term}: 
\begin{equation}
\mathcal{E}(u,v)=\mathcal{E}^{(L)}(u,v)+\mathcal{E}^{(J)}(u,v),
\label{eq:DF_dec}
\end{equation}%
where $\mathcal{E}^{(L)}$ denotes the \emph{local part} and 
\begin{equation*}
\mathcal{E}^{(J)}(u,v)=\underset{M\times M\setminus \text{diag}}{\int \int }%
(u(x)-u(y))(v(x)-v(y))dj(x,y)
\end{equation*}%
is a \emph{jump part} with jump measure $j$ defined on $M\times M\setminus $%
diag. We assume that $j$ has a density with respect to $\mu \times \mu $,
denoted by $J(x,y)$, and so the jump part $\mathcal{E}^{(J)}$ can be written
as 
\begin{equation}
\mathcal{E}^{(J)}(u,v)=\underset{M\times M}{\int \int }%
(u(x)-u(y))(v(x)-v(y))J(x,y)d\mu (y)d\mu (x).  \label{eq:J_density}
\end{equation}

For every $w\in \mathcal{F}\cap L^{\infty }$, there exists a unique positive
finite Radon measure $\Gamma (w)$ on $M,$ termed an \emph{energy measure},
such that for any $\phi \in \mathcal{F}\cap L^{\infty }$\footnote{%
Any function in $\mathcal{F}$ admits a quasi-continuous modification (cf. 
\cite[Theorem 2.1.3,p.71]{fukushima2010dirichlet}), and moreover, any energy
measure charges no set of zero capacity (cf. \cite[Lemma 3.2.4,p.127]%
{fukushima2010dirichlet}). Without loss of generality, every function in $%
\mathcal{F}$ will be replaced by its quasi-continuous modification in this
paper. Thus, the integral $\int fd\Gamma (g)$ is well-defined for any $%
f,g\in \mathcal{F}$.}, 
\begin{equation}
\int \phi d\Gamma (w)=\mathcal{E}(w\phi ,w)-\frac{1}{2}\mathcal{E}(\phi
,w^{2}),  \label{eq:g_fml}
\end{equation}%
where and in the sequel the integration $\int $ means over $M$. The energy
measure $\Gamma (w)$ can be uniquely extended to any $w\in \mathcal{F}$. For
functions $v,w\in \mathcal{F}$, the signed measure $\Gamma (v,w)$ is defined
by 
\begin{equation}
\Gamma (v,w)=\frac{1}{2}\left( \Gamma (v+w)-\Gamma (v)-\Gamma (w)\right)
\label{eq:g_fl}
\end{equation}%
(see \cite[formula (3.11)]{mosco1994composite}), and $\Gamma (v,v)\equiv
\Gamma (v)$ and 
\begin{equation*}
\mathcal{E}(v,w)=\int_{M}d\Gamma (v,w).
\end{equation*}%
For any $u,v,w\in \mathcal{F}\cap L^{\infty }$, we have by (\ref{eq:g_fml}),%
\begin{equation}
\int ud\Gamma (v,w)=\frac{1}{2}\left( \mathcal{E}(uv,w)+\mathcal{E}(v,uw)-%
\mathcal{E}(vw,u)\right) ,  \label{eq:g_fml-1}
\end{equation}%
and, from this, 
\begin{equation}
\int d\Gamma (uv,w)=\mathcal{E}(uv,w)=\int ud\Gamma (v,w)+\int vd\Gamma
(u,w).  \label{eq:L-rule}
\end{equation}%
(This can be viewed as the weak version of the product rule.)

Denote by $\Gamma _{L}(\cdot )$ the energy measure associated with \emph{%
local part} $\mathcal{E}^{(L)}$ and let $dk$ be the killing measure. Then by
Beurling-Deny's formulae (\cite[(3.2.23) p.127]{fukushima2010dirichlet}):%
\begin{equation*}
d\Gamma (u)(x)=d\Gamma _{L}(u)(x)+\left\{ \int_{M}\left( u(x)-u(y)\right)
^{2}J(x,y)d\mu (y)\right\} d\mu (x)+u^{2}(x)dk(x)
\end{equation*}%
(for this moment we do not assume that the killing term vanishes), that is,
for any $u,v\in \mathcal{F}\cap L^{\infty }$ and any non-empty open subset $%
\Omega $ of $M$,%
\begin{equation}
\int_{\Omega }u^{2}d\Gamma (v)=\int_{\Omega }u^{2}d\Gamma
_{L}(v)+\int_{\Omega \times M}u^{2}(x)\left( v(x)-v(y)\right) ^{2}J(x,y)d\mu
(y)d\mu (x)+\int_{\Omega }u^{2}v^{2}dk.  \label{11}
\end{equation}%
In particular, if there is no killing measure, then 
\begin{equation}
d\Gamma (u)(x)=d\Gamma _{L}(u)(x)+\int_{M\setminus \text{diag}}\left(
u(x)-u(y)\right) ^{2}dj(x,y).  \label{eq:G_B_def}
\end{equation}

Denote by $B\left( x,r\right) $ the open metric ball of radius $r>0$
centered at $x$. Sometimes we write $B_{r}$ for a ball of radius $r$ without
mentioning its center. Denote by $\lambda B$ a concentric ball of $B$ with
radius $\lambda r$ where $r$ is the radius of $B$. Let 
\begin{equation*}
V\left( x,r\right) :=\mu \left( B\left( x,r\right) \right)
\end{equation*}%
be the \emph{volume function}.

For a regular Dirichlet form $(\mathcal{E},\mathcal{F})$ with a jump kernel $%
J$, we define for $\rho \geq 0$ 
\begin{equation}
\mathcal{E}_{\rho }(u,v)=\mathcal{E}^{(L)}(u,v)+\int_{M}\int_{B(x,\rho
)}(u(x)-u(y))(v(x)-v(y))J(x,y)d\mu (y)d\mu (x).  \label{eq:DF_q}
\end{equation}%
It is known that $(\mathcal{E}_{\rho },\mathcal{F})$ is a closable bilinear
form and can be extended to a regular Dirichlet form $(\mathcal{E}_{\rho },%
\mathcal{F}_{\rho })$ with $\mathcal{F}\subset \mathcal{F}_{\rho }$ (see 
\cite[Section 4]{grhulau2014nonlocal}). Denote by $q_{t}(x,y)$, $%
\{Q_{t}\}_{t\geq 0}$, $\Gamma _{\rho }(\cdot )$ the heat kernel, heat
semigroup and energy measure of $(\mathcal{E}_{\rho },\mathcal{F}_{\rho })$,
respectively (we sometimes drop the superscript \textquotedblleft $\rho $"
from $q_{t}^{(\rho )}(x,y)$, $\{Q_{t}^{(\rho )}\}_{t\geq 0}$ for
simplicity). Note that if $J\equiv 0$ or if $\rho =0$, then $\left( \mathcal{%
E}_{\rho },\mathcal{F}_{\rho }\right) =(\mathcal{E},\mathcal{F})=(\mathcal{E}%
^{(L)},\mathcal{F})$, which is strongly local. Denote by 
\begin{equation}
d\Gamma _{\rho }(u)(x)=d\Gamma _{L}(u)(x)+\left\{ \int_{B(x,\rho
)}(u(x)-u(y))^{2}J(x,y)d\mu (y)\right\} d\mu (x).  \label{eq:ene_meas_split}
\end{equation}

Throughout this paper we fix some numbers $\alpha >0,\beta >0$. Fix also
some value $R_{0}\in (0,$diam$M]$ that will be used for localization of all
the hypotheses. In the sequel, the letters $C,C^{\prime },c,c^{\prime }$
denote universal positive constants which may vary at each occurrence.

Introduce the following conditions.

\begin{definitions}[Upper $\protect\alpha$-regularity]
For all $x\in M$ and all $r>0$, 
\begin{equation}
V(x,r)\leq C r^{\alpha}.  \tag*{(V\ensuremath{_\leq})}  \label{eq:vol_l}
\end{equation}
\end{definitions}

\begin{definition}[On-diagonal upper estimate]
The heat kernel $p_{t}$ exists and satisfies the on-diagonal upper estimate 
\begin{equation}
p_{t}(x,y)\leq \frac{C}{t^{\alpha /\beta }}\exp \left( R_{0}^{-\beta
}t\right)  \tag{DUE}  \label{eq:DUE}
\end{equation}%
for all $t>0$ and $\mu $-almost all $x$, $y\in M$, where $C$ is independent
of $R_{0}$ (and also of $t,x,y$).
\end{definition}

\begin{definition}[Upper estimate of non-local type]
The heat kernel $p_{t}$ exists and satisfies the off-diagonal upper estimate 
\begin{equation}
p_{t}(x,y)\leq \frac{C}{t^{\alpha /\beta }}\exp \left( R_{0}^{-\beta
}t\right) \left( 1+\frac{d(x,y)}{t^{1/\beta }}\right) ^{-(\alpha +\beta )} 
\tag{UE}  \label{eq:UE}
\end{equation}%
for all $t>0$ and $\mu $-almost all $x,y\in M$, where $C$ is independent of $%
R_{0}$ (and also of $t,x,y$).
\end{definition}

\begin{definition}[Upper bound of jump density]
The jump density $J(x,y)$ exists and admits the estimate 
\begin{equation}
J(x,y)\leq Cd(x,y)^{-(\alpha +\beta )}  \tag*{(J{\ensuremath{_\leq}})}
\label{eq:J_less}
\end{equation}%
for $\mu $-almost all $x,y\in M$.
\end{definition}

If $(\mathcal{E},\mathcal{F})$ is local, we have $J\equiv 0$ so that \ref%
{eq:J_less} is trivially satisfied. In general, condition \ref{eq:J_less}
restricts the long jumps and can be viewed as a measure of non-locality.

\begin{definition}[Upper estimate of local type]
The heat kernel $p_{t}$ exists and satisfies the off-diagonal upper estimate 
\begin{equation}
p_{t}(x,y)\leq \frac{C}{t^{\alpha /\beta }}\left( R_{0}^{-\beta }t\right)
\exp \left( -\left( \frac{d(x,y)}{ct^{1/\beta }}\right) ^{\beta /(\beta
-1)}\right)  \tag*{(UE{\ensuremath{_{\textrm{loc}} }})}  \label{eq:UE_loc}
\end{equation}%
for all $t>0$ and $\mu $-almost all $x,y\in M$, where $\beta >1$ and $C,c>0$
are independent of $R_{0},t,x,y$.
\end{definition}

Let $\Omega $ be an open subset of $M$ and $A\Subset \Omega $ be a Borel set
(where $A\Subset \Omega $ means that $A$ is precompact and its closure $%
\overline{A}\subset {\Omega }$). Recall that $\phi $ is a \emph{cutoff
function} of $(A,\Omega )$ if $\phi \in \mathcal{F}(\Omega )$, $0\leq \phi
\leq 1$ in $M$, and $\phi =1$ in an open neighborhood of $A$. We denote the
set of all cutoff functions of $(A,\Omega )$ by \emph{cutoff}$(A,\Omega )$.

\begin{definition}[Cutoff inequality on balls]
The \emph{cutoff inequality on balls} holds on $M$ if there exist constants $%
C_{1}\geq 0,C_{2}>0$ such that for any two concentric balls $B_{R},B_{R+r}$
with $0<R<R+r<R_{0}$, there exists some function $\phi \in \text{cutoff}%
(B_{R},B_{R+r})$ satisfying that 
\begin{equation}
\int_{M}u^{2}d\Gamma (\phi )\leq C_{1}\int_{M}d\Gamma (u)+\frac{C_{2}}{%
r^{\beta }}\int_{M}u^{2}d\mu  \tag{CIB}  \label{eq:CSA}
\end{equation}%
for all $u\in \mathcal{F}\cap L^{\infty }$, where $d\Gamma (u)$ is defined
by (\ref{eq:G_B_def}).
\end{definition}

Note that constants $C_{1},C_{2}$ in condition (\ref{eq:CSA}) are universal
(independent of $u,\phi ,B_{R},B_{R+r}$ and also of $R_{0}$), and the cutoff
function $\phi $ is independent of the function $u$ (but of course,
depending on the balls $B_{R},B_{R+r}$).

\begin{remark}
\RM This kind of neat condition (\ref{eq:CSA}) was first introduced by
Andres and Barlow in \cite{andres2013energy} under the framework of local
Dirichlet forms, which is called Condition $\left( CSA\right) $ -- a \emph{%
cutoff Sobolev inequality in annulus}\footnote{%
Condition ($CSA$) is actually unrelated to the classical \emph{Sobolev}
inequality.}. The condition (\ref{eq:CSA}) here is slightly weaker than
Condition\emph{\ }$\left( CSA\right) $ wherein the two integrals in the
right-hand side of (\ref{eq:CSA}) are both over the annulus (but here these
two integrals are over the whole space $M$)$.$
\end{remark}

\begin{remark}
\RM A similar condition was introduced in \cite{ghh16HE} for jump-type
Dirichlet forms, which is termed \emph{Condition }$\left( AB\right) $ named
after Andres and Barlow, and in which the cutoff function $\phi $ may depend
on $u$ (see also the generalized capacity condition ($Gcap$) stated in \cite%
{grigor2014CE}). Note that the condition (\ref{eq:CSA}) here is slightly
different from Condition\emph{\ }$\left( AB\right) $ in \cite{ghh16HE} in
that the second integral in (\ref{eq:CSA}) here is over $M$ against measure $%
d\Gamma (u)$, instead of over the larger ball $B_{R+r}$ against measure $%
\phi ^{2}d\Gamma (u)$ in \cite{ghh16HE}. Of course, it would be better to
relax condition (\ref{eq:CSA}) so that the cutoff function $\phi $ may
depend on $u$ as in \cite{ghh16HE}, \cite{grigor2014CE}-- and the reader may
consult the explanation in \cite[the remark after Definition 2.8, p.1803]%
{murugan2015davies}. More variants than (\ref{eq:CSA}) were addressed in 
\cite[Definition 2.2]{ckw16} for the nonlocal case.
\end{remark}

\begin{remark}
\RM Condition (\ref{eq:CSA}) can be easily verified with $C_{1}=0$ for
purely jump-type Dirichlet forms with $0<\beta <2$, provided that conditions %
\ref{eq:J_less}, \ref{eq:vol_l} are satisfied, using the standard \emph{bump}
function on balls (see \cite[Remark 1.7]{ckw16} or \cite[the proof of
Corollary 2.12]{ghh16HE}).
\end{remark}

The following is the main contribution of this paper.

\begin{theorem}
\label{thm:Mr2}Let $(M,d,\mu )$ be a metric measure space with precompact
balls, and let $(\mathcal{E},\mathcal{F})$ be a regular Dirichlet form in $%
L^{2}(M,\mu )$ satisfying (\ref{eq:DF_dec}) with jump kernel $J$. If
condition \ref{eq:vol_l} is satisfied, then the following implication holds: 
\begin{equation}
(\ref{eq:DUE})+(\ref{eq:CSA})+\ref{eq:J_less}\Rightarrow (\ref{eq:UE}).
\label{48}
\end{equation}%
If in addition $\beta >1$, then%
\begin{equation}
(\ref{eq:DUE})+(\ref{eq:CSA})+(J\equiv 0)\Rightarrow \ref{eq:UE_loc}.
\label{N}
\end{equation}
\end{theorem}

We apply the Davies method to prove both (\ref{48}) and (\ref{N}).
Particularly, in order to show the implication (\ref{N}), we first derive a
weaker upper bound of the heat kernel, see (\ref{C12}) below, and then
obtain \ref{eq:UE_loc} by a self-improvement technique used in \cite%
{grigor2014upper}, see Lemma \ref{UEiter} and Remark \ref{R:It} below.
Murugan and Saloff-Coste \cite{murugan2015davies} obtained a similar
implication under condition ($CSA$) introduced by Andres and Barlow, but
with a much simpler argument (without recourse to Lemma \ref{UEiter}).

As a consequence of Theorem \ref{thm:Mr2}, we have the following.

\begin{theorem}
\label{T2} Let $(M,d,\mu )$ be a metric measure space with precompact balls
and $(\mathcal{E},\mathcal{F})$ be a regular conservative Dirichlet form in $%
L^{2}$ with a jump kernel $J$. If \ref{eq:vol_l} holds, then 
\begin{equation}
(\ref{eq:UE})\Leftrightarrow (\ref{eq:DUE})+(\ref{eq:CSA})+\ref{eq:J_less}.
\label{T2a}
\end{equation}

If in addition $\beta >1$, then 
\begin{equation}
\ref{eq:UE_loc}\Leftrightarrow (\ref{eq:DUE})+(\ref{eq:CSA})+(J\equiv 0).
\label{T2b}
\end{equation}
\end{theorem}

The proof of Theorem \ref{thm:Mr2} and Theorem \ref{T2} will be given in
Section \ref{sec:HK}.

\begin{remark}
\RM For the \emph{nonlocal} case, a similar equivalence to (\ref{T2a}) was
obtained in \cite{ckw16} with (\ref{eq:CSA}) being replaced by condition CSJ$%
\left( \phi \right) $ but for more general settings equipped with doubling
measures and for more general jump kernels involving the gauge $\phi $
(noting that $\phi (r)=r^{\beta }$ for $r\geq 0$ in this paper), and also in 
\cite{ghh16HE} with condition (\ref{eq:CSA}) being replaced by condition ($%
Gcap$) or condition $\left( AB\right) $. For the \emph{local} case, a
similar equivalence to (\ref{T2b}) was obtained in \cite{andres2013energy}, 
\cite{grigor2014CE} under different variants than condition (\ref{eq:CSA}).
\end{remark}

\section[CSA]{Cutoff inequalities on balls}

\label{sec:CSA}In this section, we first derive (\ref{eq:CSA}) from
condition $\left( S\right) $-the survival estimate to be stated below. We
then state two inequalities, see (\ref{eq:CSA_strong}), (\ref%
{eq:almost_perfect}) below, which will be used in the Davies method.
Inequality (\ref{eq:CSA_strong}) can be viewed as a self-improvement of
condition (\ref{eq:CSA}).

We need the following formula.

\begin{proposition}
Let $(\mathcal{E},\mathcal{F})$ be a regular Dirichlet form in $L^{2}(M,\mu
) $. Then, for any two functions $u\in \mathcal{F}\cap L^{\infty }$, $%
\varphi \in \mathcal{F}\cap L^{\infty }$ with \textrm{supp}$(\varphi
)\subset \Omega $ for any open subset $\Omega $ of $M$, 
\begin{equation}
\int_{\Omega }u^{2}d\Gamma _{\Omega }(\varphi )\leq 2\mathcal{E}%
(u^{2}\varphi ,\varphi )+4\int_{\Omega }\varphi ^{2}d\Gamma _{\Omega }(u),
\label{eq:g_need}
\end{equation}%
where $d\Gamma _{\Omega }(u)$ is defined by%
\begin{equation}
d\Gamma _{\Omega }(u)(x)=d\Gamma _{L}(u)(x)+\int_{M\setminus \text{diag}}%
\mathbf{1}_{\Omega }(y)\left( u(x)-u(y)\right) ^{2}dj(x,y).  \label{21}
\end{equation}
\end{proposition}

\begin{proof}
We will use formula (\ref{11}). Note that 
\begin{equation*}
u^{2}\in \mathcal{F}\cap L^{\infty }\text{ \ and \ }uv,u^{2}v\in \mathcal{F}%
\cap L^{\infty }
\end{equation*}%
if $u\in \mathcal{F}\cap L^{\infty }$, $v\in \mathcal{F}\cap L^{\infty }$.

We first show that 
\begin{equation}
\int_{\Omega }u^{2}d\Gamma _{L}(\varphi )\leq 2\mathcal{E}%
^{(L)}(u^{2}\varphi ,\varphi )+4\int_{\Omega }\varphi ^{2}d\Gamma _{L}(u).
\label{L-part}
\end{equation}%
Indeed, using the Leibniz and chain rules of $d\Gamma _{L}(\cdot )$ (cf. 
\cite[Lemma 3.2.5, Theorem 3.2.2]{fukushima2010dirichlet}) and using
Cauchy-Schwarz, we have%
\begin{eqnarray*}
\int_{M}u^{2}d\Gamma _{L}(\varphi ) &=&\int_{M}d\Gamma _{L}(u^{2}\varphi
,\varphi )-2\int_{M}u\varphi d\Gamma _{L}(u,\varphi ) \\
&\leq &\mathcal{E}^{(L)}(u^{2}\varphi ,\varphi )+\frac{1}{2}%
\int_{M}u^{2}d\Gamma _{L}(\varphi )+2\int_{M}\varphi ^{2}d\Gamma _{L}(u),
\end{eqnarray*}%
which gives that%
\begin{equation*}
\int_{M}u^{2}d\Gamma _{L}(\varphi )\leq 2\mathcal{E}^{(L)}(u^{2}\varphi
,\varphi )+4\int_{M}\varphi ^{2}d\Gamma _{L}(u).
\end{equation*}%
Since $\varphi $ is supported in $\Omega $, we see that $d\Gamma
_{L}(\varphi )=0$ outside $\Omega $ (cf. \cite[formula (3.2.26) p.128]%
{fukushima2010dirichlet}), thus proving (\ref{L-part}).

Next we show that%
\begin{equation}
\int_{\Omega }u^{2}d\Gamma _{\Omega }^{(J)}(\varphi )\leq 2\mathcal{E}%
^{(J)}(u^{2}\varphi ,\varphi )+4\int_{\Omega }\varphi ^{2}d\Gamma _{\Omega
}^{(J)}(u),  \label{J-part}
\end{equation}%
where the measure $d\Gamma _{\Omega }^{(J)}$ is defined by 
\begin{equation*}
d\Gamma _{\Omega }^{(J)}(f,g)(x)=\int_{\Omega \setminus \mathrm{diag}}\left(
f(x)-f(y)\right) \left( g(x)-g(y)\right) dj(x,y).
\end{equation*}%
Indeed, noting that%
\begin{equation*}
u^{2}(x)(\varphi (x)-\varphi (y))^{2}=\left\{ \left[ (u^{2}\varphi
)(x)-(u^{2}\varphi )(y)\right] -\left[ u^{2}(x)-u^{2}(y)\right] \varphi
(y)\right\} (\varphi (x)-\varphi (y)),
\end{equation*}%
we have 
\begin{align*}
\int_{\Omega \times \Omega \setminus \mathrm{diag}}u^{2}(x)\left[ \varphi
(x)-\varphi (y)\right] ^{2}dj(x,y)=& \int_{\Omega \times \Omega \setminus 
\mathrm{diag}}\left[ (u^{2}\varphi )(x)-(u^{2}\varphi )(y)\right] (\varphi
(x)-\varphi (y))dj(x,y) \\
& -\int_{\Omega \times \Omega \setminus \mathrm{diag}}\left[
u^{2}(x)-u^{2}(y)\right] \varphi (y)(\varphi (x)-\varphi (y))dj(x,y),
\end{align*}%
which gives that%
\begin{equation}
\int_{\Omega }u^{2}d\Gamma _{\Omega }^{(J)}(\varphi )=\int_{\Omega }d\Gamma
_{\Omega }^{(J)}(u^{2}\varphi ,\varphi )-\int_{\Omega }\varphi d\Gamma
_{\Omega }^{(J)}(u^{2},\varphi ).  \label{Jp}
\end{equation}%
To estimate the last term, note that%
\begin{eqnarray*}
-\int_{\Omega }\varphi d\Gamma _{\Omega }^{(J)}(u^{2},\varphi )
&=&-\int_{\Omega \times \Omega \setminus \mathrm{diag}}(u(x)+u(y))\varphi
(y)(u(x)-u(y))(\varphi (x)-\varphi (y))dj(x,y) \\
&=&-\int_{\Omega \times \Omega \setminus \mathrm{diag}}u(x)\varphi
(y)(u(x)-u(y))(\varphi (x)-\varphi (y))dj(x,y) \\
&&-\int_{\Omega \times \Omega \setminus \mathrm{diag}}u(y)\varphi
(y)(u(x)-u(y))(\varphi (x)-\varphi (y))dj(x,y).
\end{eqnarray*}%
From this and using the Cauchy-Schwarz inequality, we derive 
\begin{align*}
-\int_{\Omega }\varphi d& \Gamma _{\Omega }^{(J)}(u^{2},\varphi ) \\
\leq & \frac{1}{4}\int_{\Omega \times \Omega \setminus \mathrm{diag}%
}u^{2}(x)(\varphi (x)-\varphi (y))^{2}dj(x,y)+\int_{\Omega \times \Omega
\setminus \mathrm{diag}}\varphi ^{2}(y)(u(x)-u(y))^{2}dj(x,y) \\
& +\frac{1}{4}\int_{\Omega \times \Omega \setminus \mathrm{diag}%
}u^{2}(y)(\varphi (x)-\varphi (y))^{2}dj(x,y)+\int_{\Omega \times \Omega
\setminus \mathrm{diag}}\varphi ^{2}(y)(u(x)-u(y))^{2}dj(x,y) \\
=& \frac{1}{2}\int_{\Omega }u^{2}d\Gamma _{\Omega }^{(J)}(\varphi
)+2\int_{\Omega }\varphi ^{2}d\Gamma _{\Omega }^{(J)}(u).
\end{align*}%
Plugging this into (\ref{Jp}), we have%
\begin{equation}
\int_{\Omega }u^{2}d\Gamma _{\Omega }^{(J)}(\varphi )\leq 2\int_{\Omega
}d\Gamma _{\Omega }^{(J)}(u^{2}\varphi ,\varphi )+4\int_{\Omega }\varphi
^{2}d\Gamma _{\Omega }^{(J)}(u).  \label{31}
\end{equation}%
As $\varphi $ vanishes outside $\Omega $, we see that%
\begin{eqnarray*}
\mathcal{E}^{(J)}(u^{2}\varphi ,\varphi ) &=&\int_{M\times M\setminus 
\mathrm{diag}}\left[ (u^{2}\varphi )(x)-(u^{2}\varphi )(y)\right] (\varphi
(x)-\varphi (y))dj(x,y) \\
&=&\int_{\Omega \times \Omega \setminus \mathrm{diag}}+2\int_{\Omega \times
\Omega ^{c}}+\int_{\Omega ^{c}\times \Omega ^{c}\setminus \mathrm{diag}%
}\cdots \\
&=&\int_{\Omega }d\Gamma _{\Omega }^{(J)}(u^{2}\varphi ,\varphi
)+2\int_{\Omega \times \Omega ^{c}}(u^{2}\varphi ^{2})(x)dj(x,y) \\
&\geq &\int_{\Omega }d\Gamma _{\Omega }^{(J)}(u^{2}\varphi ,\varphi ),
\end{eqnarray*}%
which together with (\ref{31}) implies (\ref{J-part}).

Finally, summing up (\ref{L-part}), (\ref{J-part}) we conclude from (\ref{21}%
) that (\ref{eq:g_need}) is true.
\end{proof}

We state condition $\left( S\right) $.

\begin{definition}[Survival estimate]
There exist constants $\varepsilon $, $\delta \in (0,1)$ such that, for all
balls $B$ of radius $r\in (0,R_{0})$ and for all $t^{1/\beta }\leq \delta r$%
, 
\begin{equation}
1-P_{t}^{B}1_{B}(x)\leq \varepsilon  \tag{S}  \label{eq:S}
\end{equation}%
for $\mu $-almost all $x\in \frac{1}{4}B$.
\end{definition}

\begin{lemma}
\label{lem:S_to_CSA} Let $(\mathcal{E},\mathcal{F})$ be a regular Dirichlet
form in $L^{2}(M,\mu )$. Then 
\begin{equation*}
(\ref{eq:S})\Rightarrow (\ref{eq:CSA}).
\end{equation*}
\end{lemma}

\begin{proof}
Fix $x_{0}\in M$ and set $B_{0}=B(x_{0},R)$, $B^{\prime }=B(x_{0},R+r)$ for $%
0<R<R+r<R_{0}$ and let $B^{\prime }\subset \Omega $ for any open subset $%
\Omega $ of $M$. It suffices to show that there exists some $\phi \in $%
\textrm{cutoff}$(B_{0},B^{\prime })$ such that 
\begin{equation}
\int_{\Omega }u^{2}d\Gamma _{\Omega }(\phi )\leq C_{1}\int_{\Omega }\phi
^{2}d\Gamma _{\Omega }(u)+\frac{C_{2}}{r^{\beta }}\int_{\Omega }\phi
u^{2}d\mu  \label{eq:pre-CSA}
\end{equation}%
for any $u\in \mathcal{F}\cap L^{\infty }$, where the measure $d\Gamma
_{\Omega }$ is defined by (\ref{21}), since this inequality, on taking $%
\Omega =M$ and using the fact that $\phi \leq 1$ in $M$, will imply (\ref%
{eq:CSA}).

To do this, let%
\begin{equation*}
w:=\int_{0}^{+\infty }e^{-\lambda t}P_{t}^{B^{\prime }}1_{B^{\prime }}dt,
\end{equation*}%
where $\lambda =r^{-\beta }$. It is known that 
\begin{equation}
\mathcal{E}(w,\varphi )+\lambda \int_{B^{\prime }}w\varphi d\mu
=\int_{B^{\prime }}\varphi d\mu ,  \label{eq:w_sol}
\end{equation}%
for any $\varphi \in \mathcal{F}(B^{\prime })$. By \cite[(3.6) p.1503]%
{grigor2014CE}, we have that 
\begin{equation*}
te^{-\lambda t}P_{t}^{B^{\prime }}1_{B^{\prime }}\leq w\leq r^{\beta }\;%
\text{in\ }M.
\end{equation*}%
Let $z\in B_{0}$ be any point and set $B_{z}=B(z,r)\subset B^{\prime }$. An
application of (\ref{eq:S}) with $t=\left( \delta r\right) ^{\beta }$ yields
that for almost all $x\in \frac{1}{4}B_{z}$, 
\begin{equation*}
w(x)\geq te^{-\lambda t}P_{t}^{B^{\prime }}1_{B^{\prime }}(x)\geq
te^{-\lambda t}P_{t}^{B_{z}}1_{B_{z}}(x)\geq \left( \delta r\right) ^{\beta
}e^{-\delta ^{\beta }}\left( 1-\varepsilon \right) =C_{0}^{-1}r^{\beta },
\end{equation*}%
for $C_{0}=\left[ \delta ^{\beta }e^{-\delta ^{\beta }}(1-\varepsilon )%
\right] ^{-1}>1$. Hence, 
\begin{eqnarray*}
w &\leq &r^{\beta }\;\text{in }M, \\
w &\geq &C_{0}^{-1}r^{\beta }\;\text{in}\;B_{0}.
\end{eqnarray*}%
Set $v:=C_{0}\frac{w}{r^{\beta }}$. Then $v\leq C_{0}$ in $M$, and $v\geq 1$
in $B_{0}$. Define%
\begin{equation*}
\phi =v\wedge 1\text{ in }M.
\end{equation*}%
We see that $\phi \in \mathrm{cutoff}\left( B_{0},B^{\prime }\right) $. It
suffices to show such a function $\phi $ satisfies (\ref{eq:pre-CSA}).

Indeed, using (\ref{eq:g_need}) with $\varphi =v$, 
\begin{equation}
\int_{\Omega }u^{2}d\Gamma _{\Omega }(v)\leq 2\mathcal{E}(u^{2}v,v)+4\int_{%
\Omega }v^{2}d\Gamma _{\Omega }(u).  \label{eq:s_t0}
\end{equation}%
Observe that in $M$%
\begin{equation}
C_{0}\phi =C_{0}\left( v\wedge 1\right) =\left( C_{0}v\right) \wedge
C_{0}\geq v\wedge C_{0}=v,  \label{v_E}
\end{equation}%
which gives that 
\begin{equation}
\int_{\Omega }v^{2}d\Gamma _{\Omega }(u)\leq C_{0}^{2}\int_{\Omega }\phi
^{2}d\Gamma _{\Omega }(u).  \label{eq:s_t2}
\end{equation}%
On the other hand, using (\ref{eq:w_sol}) with $\varphi =u^{2}v$ and using (%
\ref{v_E})%
\begin{eqnarray}
\mathcal{E}(u^{2}v,v) &=&\frac{C_{0}}{r^{\beta }}\mathcal{E}(u^{2}v,w) 
\notag \\
&=&\frac{C_{0}}{r^{\beta }}\left\{ \int_{B^{\prime }}u^{2}vd\mu -\lambda
\int_{B^{\prime }}\left( u^{2}v\right) wd\mu \right\}  \notag \\
&\leq &\frac{C_{0}}{r^{\beta }}\int u^{2}vd\mu \leq \frac{C_{0}^{2}}{%
r^{\beta }}\int u^{2}\phi d\mu .  \label{eq:E_es0}
\end{eqnarray}%
Thus, plugging (\ref{eq:E_es0}), (\ref{eq:s_t2}) into (\ref{eq:s_t0}), we
conclude that%
\begin{equation*}
\int_{\Omega }u^{2}d\Gamma _{\Omega }(v)\leq \frac{2C_{0}^{2}}{r^{\beta }}%
\int_{\Omega }u^{2}\phi d\mu +4C_{0}^{2}\int_{\Omega }\phi ^{2}d\Gamma
_{\Omega }(u).
\end{equation*}

Finally, using the facts that $\left\vert \phi (x)-\phi (y)\right\vert \leq
\left\vert v(x)-v(y)\right\vert $ and that $\phi (x)\leq v(x)$ for any $%
x,y\in M$ and then using \cite[formula (3.2.12), p.122]%
{fukushima2010dirichlet} and (\ref{21}), 
\begin{equation*}
\int_{\Omega }u^{2}d\Gamma _{\Omega }(\phi )\leq \int_{\Omega }u^{2}d\Gamma
_{\Omega }(v).
\end{equation*}%
Therefore, we obtain (\ref{eq:pre-CSA}) with $C_{1}=4C_{0}^{2}$, $%
C_{2}=2C_{0}^{2}$.
\end{proof}

The same result in Lemma \ref{lem:S_to_CSA} was proved in \cite%
{andres2013energy, grigor2014CE} for the local case.

We show the following two inequalities (\ref{eq:CSA_strong}) and (\ref%
{eq:almost_perfect}) by using condition (\ref{eq:CSA}).

\begin{proposition}
\label{prop:sCSA} Let $(\mathcal{E},\mathcal{F})$ be a regular Dirichlet
form in $L^{2}(M,\mu )$. Let $B_{0}=B(x_{0},R)$, $B^{\prime }=B(x_{0},R+r)$
be two balls with $0<R<R+r<R_{0}$. If conditions (\ref{eq:CSA}), \ref%
{eq:vol_l}, \ref{eq:J_less} hold, then for every positive integer $n$, there
exists some function $\phi =\phi _{n}\in $ \textrm{cutoff}$(B_{0},B^{\prime
})$ satisfying that 
\begin{equation}
\int u^{2}d\Gamma (\phi )\leq \frac{C_{3}}{n}\int d\Gamma (u)+\frac{%
C_{4}n^{\beta }}{r^{\beta }}\int u^{2}d\mu  \label{eq:CSA_strong}
\end{equation}%
for all $u\in \mathcal{F}\cap L^{\infty }$, and that%
\begin{equation}
\left\Vert \phi -\Phi \right\Vert _{\infty }\leq 1/n
\label{eq:almost_perfect}
\end{equation}%
with 
\begin{equation}
\Phi (y):=\left( \frac{R+r-d(x_{0},y)}{r}\right) _{+}\wedge 1,
\label{eq:Phi_def}
\end{equation}%
where $C_{3}\geq 1,C_{4}\geq 1$ are universal constants (independent of $%
B_{0},B^{\prime },n,u$ and $R_{0})$ and $\phi $ is independent of $u$.
\end{proposition}

\begin{proof}
Fix positive integer $n$ and for integers $0\leq k\leq n$ set $r_{k}=kr/n$, $%
B_{k}:=B(x_{0},R+r_{k})$. Define%
\begin{equation*}
U_{k}:=B_{k}\setminus B_{k-1}(1\leq k\leq n).
\end{equation*}%
For a function $u\in \mathcal{F}\cap L^{\infty }$, we apply (\ref{eq:CSA})
to each pair $\left( B_{k-1},B_{k}\right) $ ($k\geq 1$) and obtain%
\begin{equation}
\int_{M}u^{2}d\Gamma (\phi _{k})\leq C_{1}\int_{M}d\Gamma (u)+\frac{C_{2}}{%
\left( r/n\right) ^{\beta }}\int_{M}u^{2}d\mu  \label{eq:CSA0}
\end{equation}%
for some $\phi _{k}\in $\textrm{cutoff}$(B_{k-1},B_{k})$.

We define 
\begin{equation*}
\phi =\phi _{n}:=\frac{1}{n}\sum_{k=1}^{n}\phi _{k}.
\end{equation*}%
Clearly, $\phi \in $\textrm{cutoff}$\left( B_{0},B^{\prime }\right) $ and
for each $1\leq k\leq n$,%
\begin{equation*}
\frac{n-k}{n}\leq \phi =\frac{\phi _{k}+\left( \phi _{k+1}+\cdots +\phi
_{n}\right) }{n}\leq \frac{n-k+1}{n}\text{ in }U_{k}.
\end{equation*}%
On the other hand, for any $y\in U_{k}$ we have $R+r_{k-1}\leq
d(x_{0},y)<R+r_{k}$, and by definition (\ref{eq:Phi_def}) of $\Phi $, 
\begin{equation*}
\frac{n-k}{n}=1-\frac{r_{k}}{r}\leq \Phi (y)\leq 1-\frac{r_{k-1}}{r}=\frac{%
n-k+1}{n}.
\end{equation*}%
Hence, we see that (\ref{eq:almost_perfect}) holds on each $U_{k}$. Both
functions $\phi $ and $\Phi $ take values $1$ in $B_{0},$ and $0$ outside $%
B^{\prime }$, and (\ref{eq:almost_perfect}) is also true in the set $%
B_{0}\cup \left( M\setminus B^{\prime }\right) $.

It remains to prove (\ref{eq:CSA_strong}) with such choice of $\phi $.

To do this, note that, using the fact that $1_{\Omega }d\Gamma
_{L}(u_{1},u_{2})=0$ for $u_{1},u_{2}\in \mathcal{F}$ if $u_{1}$ is constant
on $\Omega $ (cf. \cite[formula (3.2.26) p. 128]{fukushima2010dirichlet}),%
\begin{equation}
\int u^{2}d\Gamma _{L}(\phi )=\frac{1}{n^{2}}\overset{n}{\underset{k=1}{\sum 
}}\int u^{2}d\Gamma _{L}(\phi _{k}).  \label{01}
\end{equation}%
On the other hand, for any $x,y\in M$,%
\begin{align}
(\phi (x)-\phi & (y))^{2}=\frac{1}{n^{2}}\left( \overset{n}{\underset{k=1}{%
\sum }}\left( \phi _{k}(x)-\phi _{k}(y)\right) \right) ^{2}  \notag \\
& =\frac{1}{n^{2}}\left\{ \overset{n}{\underset{k=1}{\sum }}\left( \phi
_{k}(x)-\phi _{k}(y)\right) ^{2}+2\overset{n-1}{\underset{k=1}{\sum }}%
\overset{n}{\underset{j=k+1}{\sum }}(\phi _{k}(x)-\phi _{k}(y))(\phi
_{j}(x)-\phi _{j}(y))\right\} .  \label{02}
\end{align}%
The last double summation contains the following terms ($j=k+1$ and $1\leq
k\leq n-1$): 
\begin{align*}
2\overset{n-1}{\underset{k=1}{\sum }}(\phi _{k}(x)-\phi _{k}(y))(\phi
_{k+1}(x)-\phi _{k+1}(y))& \leq \overset{n-1}{\underset{k=1}{\sum }}(\phi
_{k}(x)-\phi _{k}(y))^{2}+\overset{n-1}{\underset{k=1}{\sum }}(\phi
_{k+1}(x)-\phi _{k+1}(y))^{2} \\
& \leq 2\overset{n}{\underset{k=1}{\sum }}(\phi _{k}(x)-\phi _{k}(y))^{2},
\end{align*}%
where we have used the Cauchy-Schwarz. From this, we obtain from (\ref{02})
that%
\begin{equation*}
\left( \phi (x)-\phi (y)\right) ^{2}\leq \frac{3}{n^{2}}\overset{n}{\underset%
{k=1}{\sum }}\left( \phi _{k}(x)-\phi _{k}(y)\right) ^{2}+\frac{2}{n^{2}}%
\overset{n-2}{\underset{k=1}{\sum }}\overset{n}{\underset{j=k+2}{\sum }}%
(\phi _{k}(x)-\phi _{k}(y))(\phi _{j}(x)-\phi _{j}(y)).
\end{equation*}%
Multiplying by $u^{2}(x)J(x,y)$ then integrating over $M\times M$ on both
sides, we obtain that%
\begin{align}
\int_{M\times M}& u^{2}(x)(\phi (x)-\phi (y))^{2}J(x,y)d\mu (y)d\mu (x) 
\notag \\
\leq & \frac{3}{n^{2}}\overset{n}{\underset{k=1}{\sum }}\int_{M\times
M}u^{2}(x)\left( \phi _{k}(x)-\phi _{k}(y)\right) ^{2}J(x,y)d\mu (y)d\mu (x)
\notag \\
& +\frac{2}{n^{2}}\overset{n-2}{\underset{k=1}{\sum }}\overset{n}{\underset{%
j=k+2}{\sum }}\int_{M\times M}u^{2}(x)(\phi _{k}(x)-\phi _{k}(y))(\phi
_{j}(x)-\phi _{j}(y))J(x,y)d\mu (y)d\mu (x).  \label{03}
\end{align}%
Noting that for any $1\leq k\leq n-2$ and any $k+2\leq j\leq n$ 
\begin{equation*}
\phi _{j}\phi _{k}=\phi _{k}\text{ in }M,
\end{equation*}%
we have that for any $x,y\in M$ 
\begin{align*}
(\phi _{k}(x)-\phi _{k}(y))(\phi _{j}(x)-\phi _{j}(y))& =\phi _{k}(x)-\phi
_{k}(x)\phi _{j}(y)-\phi _{k}(y)\phi _{j}(x)+\phi _{k}(y) \\
& =\phi _{k}(x)(1-\phi _{j}(y))+\phi _{k}(y)(1-\phi _{j}(x)).
\end{align*}%
Plugging this into (\ref{03}) and then summing up with (\ref{01}), we obtain
that 
\begin{align}
\int u^{2}d\Gamma (\phi )=& \int u^{2}d\Gamma _{L}(\phi )+\int_{M\times
M}u^{2}(x)(\phi (x)-\phi (y))^{2}J(x,y)d\mu (y)d\mu (x)  \notag \\
\leq & \overset{n}{\underset{k=1}{\sum }}\left\{ \frac{1}{n^{2}}\int
u^{2}d\Gamma _{L}(\phi _{k})+\frac{3}{n^{2}}\int_{M\times M}u^{2}(x)\left(
\phi _{k}(x)-\phi _{k}(y)\right) ^{2}J(x,y)d\mu (y)d\mu (x)\right\}  \notag
\\
& +\frac{2}{n^{2}}\overset{n-2}{\underset{k=1}{\sum }}\overset{n}{\underset{%
j=k+2}{\sum }}\int_{M\times M}u^{2}(x)\phi _{k}(x)(1-\phi _{j}(y))J(x,y)d\mu
(y)d\mu (x)  \notag \\
& +\frac{2}{n^{2}}\overset{n-2}{\underset{k=1}{\sum }}\overset{n}{\underset{%
j=k+2}{\sum }}\int_{M\times M}u^{2}(x)\phi _{k}(y)(1-\phi _{j}(x))J(x,y)d\mu
(y)d\mu (x)  \notag \\
:=& I_{1}+2I_{2}+2I_{3}.  \label{04}
\end{align}%
We will estimate $I_{1}$, $I_{2}$, $I_{3}$ separately.

For the term $I_{1}$, we apply (\ref{eq:CSA0}) to obtain%
\begin{align}
I_{1}& \leq \frac{3}{n^{2}}\overset{n}{\underset{k=1}{\sum }}\int
u^{2}d\Gamma (\phi _{k})\leq \frac{3}{n^{2}}\overset{n}{\underset{k=1}{\sum }%
}\left\{ C_{1}\int d\Gamma (u)+\frac{C_{2}}{\left( r/n\right) ^{\beta }}\int
u^{2}d\mu \right\}  \notag \\
& \leq \frac{3C_{1}}{n}\int d\Gamma (u)+\frac{3C_{2}n^{\beta -1}}{r^{\beta }}%
\int u^{2}d\mu .  \label{06}
\end{align}

For the term $I_{2}$, observe that for any $1\leq k\leq n-2$ and $k+2\leq
j\leq n$,%
\begin{equation}
\mathrm{dist}\left( \text{\textrm{supp}}(\phi _{k}),\text{\textrm{supp}}%
(1-\phi _{j})\right) \geq \mathrm{dist}\left( B_{k},B_{j-1}^{c}\right) \geq
r/n.  \label{dist}
\end{equation}%
Recall that by \ref{eq:J_less}, \ref{eq:vol_l}, we have that 
\begin{equation}
\int_{B(x,\rho )^{c}}J(x,y)d\mu (y)\leq C\rho ^{-\beta }  \label{60}
\end{equation}%
for $\mu $-almost all $x\in M$ (cf. \cite[Proof of Proposition 4.7]%
{grhulau2014nonlocal}). From this and using (\ref{dist}), we have%
\begin{align*}
\int_{M\times M}u^{2}(x)\phi _{k}(x)(1-\phi _{j}(y))& J(x,y)d\mu (y)d\mu (x)
\\
=& \int_{B_{k}\times B_{j-1}^{c}}u^{2}(x)\phi _{k}(x)(1-\phi
_{j}(y))J(x,y)d\mu (y)d\mu (x) \\
\leq & \int_{B_{k}\times B_{j-1}^{c}}u^{2}(x)J(x,y)d\mu (y)d\mu (x) \\
\leq & \frac{C^{\prime }}{\left( r/n\right) ^{\beta }}\int_{B_{k}}u^{2}(x)d%
\mu (x).
\end{align*}%
Hence,%
\begin{equation}
I_{2}\leq \frac{1}{n^{2}}\overset{n-2}{\underset{k=1}{\sum }}\overset{n}{%
\underset{j=k+2}{\sum }}\frac{C^{\prime }}{\left( r/n\right) ^{\beta }}%
\int_{B_{k}}u^{2}(x)d\mu (x)\leq \frac{C^{\prime }n^{\beta }}{r^{\beta }}%
\int u^{2}d\mu .  \label{07}
\end{equation}%
Similarly, we have that, using (\ref{dist}), (\ref{60}),%
\begin{align*}
\int_{M\times M}u^{2}(x)\phi _{k}(y)(1-\phi _{j}(x))J(x,y)d\mu (y)d\mu (x)&
\leq \int_{B_{j-1}^{c}\times B_{k}}u^{2}(x)J(x,y)d\mu (y)d\mu (x) \\
& \leq \frac{C^{\prime }}{\left( r/n\right) ^{\beta }}%
\int_{B_{j-1}^{c}}u^{2}(x)d\mu (x),
\end{align*}%
which gives that%
\begin{equation}
I_{3}\leq \frac{1}{n^{2}}\overset{n-2}{\underset{k=1}{\sum }}\overset{n}{%
\underset{j=k+2}{\sum }}\frac{C^{\prime }}{\left( r/n\right) ^{\beta }}%
\int_{B_{j-1}^{c}}u^{2}(x)d\mu (x)\leq \frac{C^{\prime }n^{\beta }}{r^{\beta
}}\int u^{2}d\mu .  \label{08}
\end{equation}%
Therefore, plugging (\ref{06}), (\ref{07}) and (\ref{08}) into (\ref{04}),
we conclude that%
\begin{equation*}
\int u^{2}d\Gamma (\phi )\leq \frac{3C_{1}}{n}\int d\Gamma (u)+\frac{\left(
3C_{2}+4C^{\prime }\right) n^{\beta }}{r^{\beta }}\int u^{2}d\mu ,
\end{equation*}%
thus proving (\ref{eq:CSA_strong}). The proof is complete.
\end{proof}

\begin{remark}
\RM A self-improvement of condition $\left( AB\right) $ for the jump type
(nonlocal) Dirichlet form is addressed in \cite[Lemma 2.9]{ghh16HE} by using
a much more complicated cutoff function constructed in \cite[Lemma 5.1]%
{andres2013energy}. The reason is that in \cite[Lemma 2.9]{ghh16HE} one
needs to consider the integral $\int \phi ^{2}d\Gamma (u)$- a smaller
integral involving the function $\phi ^{2}$, incurring a certain amount of
troubles. See also \cite[the proof of Proposition 2.4]{ckw16} for a
self-improvement of condition CSJ$\left( \phi \right) $.
\end{remark}

\begin{remark}
\RM If the Dirichlet form is \emph{strongly local}, the following sharper
self-improvement than (\ref{eq:CSA_strong}) was proved in \cite[Lemma 2.1]%
{murugan2015davies}: setting $U=B_{R+r}\setminus B_{R}$, 
\begin{equation*}
\int_{U}u^{2}d\Gamma (\phi )\leq \frac{C_{3}}{n}\int_{U}d\Gamma (u)+\frac{%
C_{4}n^{\beta /2\,-1}}{r^{\beta }}\int_{U}u^{2}d\mu ,
\end{equation*}%
which contains the better factor $n^{\beta /2\,-1}$ (instead of $n^{\beta }$
in (\ref{eq:CSA_strong}) above), but starting from a stronger assumption ($%
CSA$) introduced in \cite{andres2013energy}. For the non-local Dirichlet
form, this factor does not play a role as we will see in the proof of
Theorem \ref{thm:Mr2} below.
\end{remark}

\section[Heat kernel]{Off-diagonal upper bound}

\label{sec:HK}

In this section, we prove Theorem \ref{thm:Mr2} and Theorem \ref{T2}. To
prove Theorem \ref{thm:Mr2}, we need to obtain upper estimate of the heat
kernel $q_{t}^{(\rho )}(x,y)$ associated with the \emph{truncated} Dirichlet
form $(\mathcal{E}_{\rho },\mathcal{F}_{\rho })$ defined by (\ref{eq:DF_q})
for any $0<\rho <\infty $. This can be done by carrying out Davies'
perturbation method. Note that the form $(\mathcal{E},\mathcal{F})$ or $(%
\mathcal{E}_{\rho },\mathcal{F}_{\rho })$ may not be conservative at this
stage.

For any regular Dirichlet form $(\mathcal{E},\mathcal{F})$, recall the
following identity (cf. \cite[formula (4.5.7), p.181]{fukushima2010dirichlet}%
): for any $u,v\in \mathcal{F}$,\textbf{\ }%
\begin{align}
\mathcal{E}(u,v)=\underset{t\rightarrow 0+}{\lim }\mathcal{E}^{(t)}(u,v):=%
\underset{t\rightarrow 0+}{\lim }& \Bigg\{\frac{1}{2t}\int_{M\times
M}(u(x)-u(y))(v(x)-v(y))P_{t}(x,dy)d\mu (x)  \notag \\
& +\frac{1}{t}\int_{M}uv(1-P_{t}1)d\mu \Bigg\},  \label{eq:t_fml}
\end{align}%
where $P_{t}(x,dy)$ is the transition function. Using this, we have that for
any three measurable functions $u,v,w$ such that all functions $u,vw,uv,w$
belong to $\mathcal{F}$, 
\begin{align}
\mathcal{E}(u,vw)=& \mathcal{E}(uv,w)  \notag \\
& +\underset{t\rightarrow 0+}{\lim }\frac{1}{2t}\int_{M\times
M}(u(x)w(y)-u(y)w(x))(v(x)-v(y))P_{t}(x,dy)d\mu (x),  \label{eq:tr_use}
\end{align}%
and that, using (\ref{eq:g_fml}), 
\begin{equation}
\int ud\Gamma (v)=\underset{t\rightarrow 0+}{\lim }\int ud\Gamma ^{(t)}(v)
\label{eq:tr_fml}
\end{equation}%
for any $u,v\in \mathcal{F}\cap L^{\infty },$ where%
\begin{equation*}
\int ud\Gamma ^{(t)}(v):=\frac{1}{2t}\left\{ \int_{M\times
M}u(x)(v(x)-v(y))^{2}P_{t}(x,dy)d\mu (x)+\int_{M}uv^{2}(1-P_{t}1)d\mu
\right\} .
\end{equation*}

For any function $f$ and any number $\rho >0$, set%
\begin{equation}
\text{\textrm{osc}}(f,\rho ):=\underset{d(x,y)\leq \rho }{\underset{x,y\in M}%
{\sup }}\left\vert f(y)-f(x)\right\vert .  \label{eq:osc}
\end{equation}%
The following lemma is motivated by \cite[Theorem 3.9]{carlen1987upper}, 
\cite[Lemma 3.4]{murugan2015heat}.

\begin{lemma}
\label{lem:DV_1} Let $(\mathcal{E}_{\rho },\mathcal{F}_{\rho })$ be a
regular Dirichlet form defined in (\ref{eq:DF_q}) for $0<\rho <\infty $ with
energy measure $d\Gamma _{\rho }(\cdot )$. Then 
\begin{equation}
\mathcal{E}_{\rho }(e^{-\psi }f,e^{\psi }f^{2p-1})\geq \frac{1}{2p}\mathcal{E%
}_{\rho }(f^{p})-9p\Lambda _{\psi }\int_{M}f^{2p}d\Gamma _{\rho }(\psi )
\label{eq:Dv0}
\end{equation}%
for any $\psi \in \mathcal{F}\cap L^{\infty }$, any non-negative $f\in 
\mathcal{F}\cap L^{\infty }$ and any $p\geq 1$, where%
\begin{equation}
\Lambda _{\psi }=%
\begin{cases}
1, & J\equiv 0, \\ 
e^{2\text{\textrm{osc}}(\psi ,\rho )}, & J\neq 0.%
\end{cases}
\label{43}
\end{equation}
\end{lemma}

\begin{proof}
We note that $e^{\psi }-1\in \mathcal{F}\cap L^{\infty }$ by using (\ref%
{eq:t_fml}) and the elementary inequality 
\begin{equation}
(e^{a}-1)^{2}\leq e^{2|a|}a^{2}  \label{eq:ele_a}
\end{equation}%
for all $a\in \mathbb{R}$. It follows that both functions $e^{\psi }g$ and $%
e^{-\psi }g$ belong to $\mathcal{F}\cap L^{\infty }$ if $g\in \mathcal{F}%
\cap L^{\infty }$. By symmetry $Q_{t}(x,dy)d\mu (x)=Q_{t}(y,dx)d\mu (y)$ for
the transition function $Q_{t}(x,dy)$ associated with the form $(\mathcal{E}%
_{\rho },\mathcal{F}_{\rho })$, 
\begin{equation*}
m_{t}(dx,dy):=\frac{1}{2t}Q_{t}(x,dy)d\mu (x)=\frac{1}{2t}Q_{t}(y,dx)d\mu
(y).
\end{equation*}%
Applying (\ref{eq:tr_use}) with $u=e^{-\psi }f$, $v=e^{\psi }$, $w=f^{2p-1}$
and $\mathcal{E}$ being replaced by $\mathcal{E}_{\rho }$, we have%
\begin{eqnarray}
\mathcal{E}_{\rho }(e^{-\psi }f,e^{\psi }f^{2p-1}) &=&\mathcal{E}_{\rho
}(f,f^{2p-1})  \notag \\
&&+\underset{t\downarrow 0+}{\lim }\Bigg\{\int_{M\times M}\left[ \left(
e^{-\psi }f\right) (x)f^{2p-1}(y)-\left( e^{-\psi }f\right) (y)f^{2p-1}(x)%
\right]  \notag \\
&&\qquad \qquad \times \left( e^{\psi (x)}-e^{\psi (y)}\right) m_{t}(dx,dy)%
\Bigg\}:=I_{1}+I_{2}.  \label{eq:Dv-2}
\end{eqnarray}%
For $I_{1}$, we have%
\begin{equation}
I_{1}=\mathcal{E}_{\rho }(f,f^{2p-1})\geq \frac{2p-1}{p^{2}}\mathcal{E}%
_{\rho }(f^{p})  \label{st0}
\end{equation}%
(cf. \cite[formulae (3.17)]{carlen1987upper} that is valid for any regular
Dirichlet form by using (\ref{eq:t_fml})).

To estimate $I_{2}$, note that 
\begin{align}
I_{2}=\underset{t\downarrow 0+}{\lim }& \left\{ \int_{M\times M}\left(
f^{2p}(y)-f^{2p}(x)\right) e^{-\psi (x)}\left( e^{\psi (x)}-e^{\psi
(y)}\right) m_{t}(dx,dy)\right.  \notag \\
& \;+\int_{M\times M}f^{2p}(x)\left( e^{-\psi (x)}-e^{-\psi (y)}\right)
\left( e^{\psi (x)}-e^{\psi (y)}\right) m_{t}(dx,dy)  \notag \\
& \left. \;+2\int_{M\times M}f^{2p-1}(y)\left( f(x)-f(y)\right) e^{\psi
(y)}\left( e^{-\psi (y)}-e^{-\psi (x)}\right) m_{t}(dx,dy)\right\} .
\label{I2}
\end{align}%
From this and using the Cauchy-Schwarz and the following elementary
inequalities%
\begin{equation*}
\int f^{2p-2}d\Gamma _{\rho }(f)\leq \mathcal{E}_{\rho }(f^{2p-1},f)\leq 
\mathcal{E}_{\rho }(f^{p})
\end{equation*}%
(see \cite[formulas (3.16), (3.17)]{carlen1987upper}), we have 
\begin{align}
I_{2}\geq & -\sqrt{\mathcal{E}_{\rho }(f^{p})}\left( \sqrt{\int
f^{2p}e^{2\psi }d\Gamma _{\rho }(e^{-\psi }-1)}+\sqrt{\int f^{2p}e^{-2\psi
}d\Gamma _{\rho }(e^{\psi }-1)}\right)  \notag \\
& -\left( \int f^{2p}e^{2\psi }d\Gamma _{\rho }(e^{-\psi }-1)\cdot \int
f^{2p}e^{-2\psi }d\Gamma _{\rho }(e^{\psi }-1)\right) ^{1/2}  \notag \\
& -2\left( \mathcal{E}_{\rho }(f^{p})\int f^{2p}e^{2\psi }d\Gamma _{\rho
}(e^{-\psi }-1)\right) ^{1/2}.  \label{eq:I2}
\end{align}%
We further estimate $I_{2}$ by using the following fact:%
\begin{equation}
\max \left\{ \int f^{2p}e^{2\psi }d\Gamma _{\rho }(e^{-\psi }-1),\int
f^{2p}e^{-2\psi }d\Gamma _{\rho }(e^{\psi }-1)\right\} \leq \Lambda _{\psi
}\int f^{2p}d\Gamma _{\rho }(\psi ),  \label{eq:os_es}
\end{equation}%
and indeed, this fact can be proved by noting that 
\begin{equation*}
e^{-2\psi }d\Gamma _{L}(e^{\psi }-1)=d\Gamma _{L}(\psi )=e^{2\psi }d\Gamma
_{L}(e^{-\psi }-1)
\end{equation*}%
from the chain rule for the energy measure $d\Gamma _{L}(\cdot )$ (cf. \cite[%
Theorem 3.2.2]{fukushima2010dirichlet}), and that%
\begin{equation*}
e^{2\psi (x)}(e^{-\psi (x)}-e^{-\psi (y)})^{2}\leq e^{2\text{\textrm{osc}}%
(\psi ,\rho )}\left\vert \psi (x)-\psi (y)\right\vert ^{2}
\end{equation*}%
for any $x,y$ with $d(x,y)\leq \rho $.

Then, plugging (\ref{eq:os_es}) into (\ref{eq:I2}) and then using the
elementary inequality $4ab\leq \frac{a^{2}}{2p}+8pb^{2}$ for $a,b>0$, we
obtain 
\begin{equation}
I_{2}\geq -\frac{1}{2p}\mathcal{E}_{\rho }(f^{p})-9p\Lambda _{\psi }\int
f^{2p}d\Gamma _{\rho }(\psi ).  \label{eq:I2_e}
\end{equation}%
Finally, plugging (\ref{st0}), (\ref{eq:I2_e}) into (\ref{eq:Dv-2}), we
conclude that, using $\frac{2p-1}{p^{2}}\geq \frac{1}{p}$,%
\begin{equation*}
\mathcal{E}_{\rho }(e^{-\psi }f,e^{\psi }f^{2p-1})\geq \frac{1}{2p}\mathcal{E%
}_{\rho }(f^{p})-9p\Lambda _{\psi }\int f^{2p}d\Gamma _{\rho }(\psi ),
\end{equation*}%
thus proving (\ref{eq:Dv0}), as desired.
\end{proof}

We estimate the last term in (\ref{eq:Dv0}) by using the cutoff inequality
developed in Section \ref{sec:CSA}. For $0<\eta <1$, set 
\begin{equation}
c_{1}(\eta )=%
\begin{cases}
0, & J\equiv 0, \\ 
2(\beta +1)\left( \eta +2\eta ^{2}\right) , & J\neq 0.%
\end{cases}
\label{eq:c1_Def}
\end{equation}

\begin{lemma}
\label{lem:Dv_1} Let $B_{R},B_{R+r}$ be two concentric balls in $M$ with $%
0<R<R+r<R_{0}$ and let $(\mathcal{E}_{\rho },\mathcal{F}_{\rho })$ be the
truncated Dirichlet form defined in (\ref{eq:DF_q}) with $0<\rho <r$. Assume
all the conditions \ref{eq:J_less}, \ref{eq:vol_l} and (\ref{eq:CSA}) are
satisfied by the form $(\mathcal{E},\mathcal{F})$. Then for any $p\geq 1$
and any $\lambda \geq \eta ^{-1}$ with $\eta :=\rho /r$, and for any
non-negative $f\in \mathcal{F}\cap L^{\infty }$, there exists some function $%
\phi =\phi _{p,\lambda }\in $\textrm{cutoff}$(B_{R},B_{R+r})$ (independent
of $f$) such that%
\begin{equation}
\mathcal{E}_{\rho }(e^{-\lambda \phi }f,e^{\lambda \phi }f^{2p-1})\geq \frac{%
1}{4p}\mathcal{E}(f^{p})-TC_{0}p^{2\beta +1}\lambda ^{2\beta +2}\int
f^{2p}d\mu ,  \label{eq:R}
\end{equation}%
and that%
\begin{equation}
\left\Vert \phi -\Phi \right\Vert _{\infty }\leq \frac{1}{\left( 6\lambda
p\right) ^{2}}<\frac{1}{\lambda p},  \label{eq:perf_func}
\end{equation}%
where $C_{0}$ is some universal constant independent of $B_{R},B_{R+r},\rho
,p,\lambda ,R_{0}$ and functions $\phi ,f$, and where $\Phi $ is given by (%
\ref{eq:Phi_def}), and%
\begin{equation}
T=%
\begin{cases}
1/r^{\beta }, & J\equiv 0, \\ 
e^{c_{1}(\eta )\lambda }/\rho ^{\beta }, & J\neq 0,%
\end{cases}
\label{tt}
\end{equation}%
with $c_{1}(\eta )$ given by (\ref{eq:c1_Def}).
\end{lemma}

\begin{remark}
\label{R11}\textrm{We will see from the proof below that the energy $%
\mathcal{E}_{\rho }$(e$^{\lambda \phi }$f,e$^{-\lambda \phi }$f$^{2p-1}$)
has the same lower bound as in (\ref{eq:R}). }
\end{remark}

\begin{proof}
Applying Lemma \ref{lem:DV_1} with $\psi =\lambda \phi $, we have%
\begin{equation}
\mathcal{E}_{\rho }(e^{-\lambda \phi }f,e^{\lambda \phi }f^{2p-1})\geq \frac{%
1}{2p}\mathcal{E}_{\rho }(f^{p},f^{p})-9p\lambda ^{2}\Lambda _{\phi
}^{\lambda }\int f^{2p}d\Gamma _{\rho }(\phi ),  \label{eq:st0}
\end{equation}%
for any $\phi \in \text{cutoff}(B_{R},B_{R+r})$, any $0\leq f\in \mathcal{F}%
\cap L^{\infty }$ and any $p\geq 1$, $\lambda >0$, where $\Lambda _{\phi }$
is given by (\ref{43}) with $\psi $ being replaced by $\phi $. Using (\ref%
{60}) and \cite[Proposition 4.1]{grhulau2014nonlocal}, we have%
\begin{equation}
\mathcal{E}(f^{p})-\mathcal{E}_{\rho }(f^{p})\leq 4\int f^{2p}d\mu \cdot 
\underset{x\in M}{\sup }\left\{ \int_{B(x,\rho )^{c}}J(x,y)d\mu (y)\right\}
\leq \frac{C_{5}}{\rho ^{\beta }}\int f^{2p}d\mu ,  \label{r}
\end{equation}%
where $C_{5}\geq 0$ is some universal constant independent of $f,p,\rho $
(noting that $C_{5}=0$ if $J\equiv 0$).

Plugging this into (\ref{eq:st0}) we obtain%
\begin{equation}
\mathcal{E}_{\rho }(e^{-\lambda \phi }f,e^{\lambda \phi }f^{2p-1})\geq \frac{%
1}{2p}\left\{ \mathcal{E}(f^{p})-\frac{C_{5}}{\rho ^{\beta }}\int f^{2p}d\mu
\right\} -9p\lambda ^{2}\Lambda _{\phi }^{\lambda }\int f^{2p}d\Gamma _{\rho
}(\phi ).  \label{2}
\end{equation}

We further estimate the energy $\mathcal{E}_{\rho }(e^{-\lambda \phi
}f,e^{\lambda \phi }f^{2p-1})$ starting from (\ref{2}) by using a
self-improvement (\ref{eq:CSA_strong}) of condition (\ref{eq:CSA}).

To do this, we claim that%
\begin{eqnarray}
\mathcal{E}_{\rho }(e^{-\lambda \phi }f,e^{\lambda \phi }f^{2p-1}) &\geq &%
\frac{1}{4p}\mathcal{E}(f^{p})-\frac{C_{5}}{2p\rho ^{\beta }}\int f^{2p}d\mu
\notag \\
&&-9p\lambda ^{2}e^{2\lambda c_{2}(\eta )}\frac{C_{4}n^{2\beta }}{r^{\beta }}%
\int f^{2p}d\mu \text{,}  \label{eq:E_precise}
\end{eqnarray}%
where $c_{2}(\eta )$ is defined by 
\begin{equation}
c_{2}(\eta )=%
\begin{cases}
0, & J\equiv 0, \\ 
\eta +2\eta ^{2}, & J\neq 0.%
\end{cases}
\label{eq:c2_eta}
\end{equation}

We distinguish two cases.

\emph{Case }$J\neq 0$\emph{.} Applying (\ref{eq:CSA_strong}) with $u=f^{p}$
and $n$ being replaced by $n^{2}$, we have that for each integer $n\geq 1$,
there exists $\phi :=\phi _{n}\in \text{cutoff}(B_{R},B_{R+r})$ satisfying
that%
\begin{equation}
\int f^{2p}d\Gamma _{\rho }(\phi )\leq \int f^{2p}d\Gamma (\phi )\leq \frac{%
C_{3}}{n^{2}}\mathcal{E}(f^{p})+\frac{C_{4}n^{2\beta }}{r^{\beta }}%
\int_{M}f^{2p}d\mu ,  \label{eq:CSN}
\end{equation}%
and that%
\begin{equation}
\left\Vert \phi -\Phi \right\Vert _{\infty }\leq \frac{1}{n^{2}}.
\label{eq:a_per}
\end{equation}

By definition (\ref{eq:Phi_def}) of $\Phi $, we see that%
\begin{equation}
\text{\textrm{osc}}(\Phi ,\rho )\leq \sup_{d(x,y)\leq \rho }\frac{d(x,y)}{r}%
\leq \frac{\rho }{r}=\eta .  \label{3}
\end{equation}%
Thus, we see from (\ref{eq:a_per}), (\ref{3}), (\ref{eq:c2_eta}) that%
\begin{equation*}
\text{\textrm{osc}}(\phi ,\rho )\leq \text{\textrm{osc}}(\Phi ,\rho )+\frac{2%
}{n^{2}}\leq \eta +2\eta ^{2}=c_{2}(\eta )
\end{equation*}%
provided that%
\begin{equation}
n\geq \frac{1}{\eta }.  \label{4}
\end{equation}%
This implies by (\ref{43}) that%
\begin{equation*}
\Lambda _{\phi }^{\lambda }=e^{2\lambda \text{\textrm{osc}}(\phi ,\rho
)}\leq e^{2\lambda c_{2}(\eta )}.
\end{equation*}

Therefore, using this, we obtain from (\ref{2}), (\ref{eq:CSN}) that under (%
\ref{4}),%
\begin{eqnarray}
\mathcal{E}_{\rho }(e^{-\lambda \phi }f,e^{\lambda \phi }f^{2p-1}) &\geq &%
\frac{1}{2p}\left\{ \mathcal{E}(f^{p})-\frac{C_{5}}{\rho ^{\beta }}\int
f^{2p}d\mu \right\}  \notag \\
&&-9p\lambda ^{2}e^{2\lambda c_{2}(\eta )}\left\{ \frac{C_{3}}{n^{2}}%
\mathcal{E}(f^{p})+\frac{C_{4}n^{2\beta }}{r^{\beta }}\int_{M}f^{2p}d\mu
\right\}  \notag \\
&=&\left\{ \frac{1}{2p}-9p\lambda ^{2}e^{2\lambda c_{2}(\eta )}\frac{C_{3}}{%
n^{2}}\right\} \mathcal{E}(f^{p})  \notag \\
&&-\left\{ \frac{C_{5}}{2p\rho ^{\beta }}+9p\lambda ^{2}e^{2\lambda
c_{2}(\eta )}\frac{C_{4}n^{2\beta }}{r^{\beta }}\right\} \int f^{2p}d\mu .
\label{5}
\end{eqnarray}%
Choose the least integer $n\geq 1$ such that%
\begin{equation*}
\frac{1}{2p}-9p\lambda ^{2}e^{2\lambda c_{2}(\eta )}\frac{C_{3}}{n^{2}}\geq 
\frac{1}{4p},
\end{equation*}%
that is,%
\begin{equation}
n=\left\lceil 6p\lambda \exp (\lambda c_{2}(\eta ))\sqrt{C_{3}}\right\rceil .
\label{eq:n_choice}
\end{equation}%
With such choice of $n$, condition (\ref{4}) is satisfied by using the
assumption that $\lambda \geq \eta ^{-1}$, since $C_{3}\geq 1$ and%
\begin{equation*}
n=\left\lceil 6p\lambda \exp (\lambda c_{2}(\eta ))\sqrt{C_{3}}\right\rceil
\geq 6p\lambda >\frac{1}{\eta }.
\end{equation*}%
From this and using (\ref{5}), we obtain that (\ref{eq:E_precise}) holds for 
$J\neq 0$.

\textit{Case }$J\equiv 0$\textit{.} It is not difficult to see from above
that (\ref{eq:E_precise}) also follows from (\ref{2}) with $C_{5}=0$ and $%
c_{2}(\eta )=0$ if $J\equiv 0$, since $\Lambda _{\phi }\equiv 1$.

Therefore, inequality (\ref{eq:E_precise}) holds, and our claim is true.

Noting that by (\ref{eq:n_choice})%
\begin{equation*}
n\leq 6p\lambda \exp (\lambda c_{2}(\eta ))\sqrt{C_{3}}+1\leq 12p\lambda
\exp (\lambda c_{2}(\eta ))\sqrt{C_{3}},
\end{equation*}%
we have that, using the fact that $c_{1}(\eta )=2(\beta +1)c_{2}(\eta )$ by (%
\ref{eq:c1_Def}), (\ref{eq:c2_eta}), 
\begin{eqnarray}
9p\lambda ^{2}e^{2\lambda c_{2}(\eta )}\frac{C_{4}n^{2\beta }}{r^{\beta }}
&\leq &9p\lambda ^{2}e^{2\lambda c_{2}(\eta )}\frac{C_{4}\left( 12p\lambda
\exp (\lambda c_{2}(\eta ))\sqrt{C_{3}}\right) ^{2\beta }}{r^{\beta }} 
\notag \\
&=&C_{6}p^{2\beta +1}\lambda ^{2\beta +2}\frac{\exp (2(\beta +1)c_{2}(\eta
)\lambda )}{r^{\beta }}  \notag \\
&=&C_{6}p^{2\beta +1}\lambda ^{2\beta +2}\frac{\exp (c_{1}(\eta )\lambda )}{%
r^{\beta }},  \label{eq:L_term}
\end{eqnarray}%
where $C_{6}=9\times 12^{2\beta }C_{4}C_{3}^{\beta }$. Plugging (\ref%
{eq:L_term}) into (\ref{eq:E_precise}), we see that%
\begin{eqnarray}
\mathcal{E}_{\rho }(e^{-\lambda \phi }f,e^{\lambda \phi }f^{2p-1}) &\geq &%
\frac{1}{4p}\mathcal{E}(f^{p})-\frac{C_{5}}{2p\rho ^{\beta }}\int f^{2p}d\mu
\notag \\
&&-C_{6}p^{2\beta +1}\lambda ^{2\beta +2}\frac{\exp (c_{1}(\eta )\lambda )}{%
r^{\beta }}\int f^{2p}d\mu \text{,}  \label{Energy}
\end{eqnarray}%
which gives that%
\begin{equation*}
\mathcal{E}_{\rho }(e^{-\lambda \phi }f,e^{\lambda \phi }f^{2p-1})\geq \frac{%
1}{4p}\mathcal{E}(f^{p})-p^{2\beta +1}\lambda ^{2\beta +2}\left[ \frac{C_{5}%
}{\rho ^{\beta }}+\frac{C_{6}\exp (c_{1}(\eta )\lambda )}{r^{\beta }}\right]
\int f^{2p}d\mu ,
\end{equation*}%
thus, proving (\ref{eq:R}) by setting $C_{0}=C_{5}+C_{6}$, with $T$ given by
(\ref{tt}).

Finally, inequality (\ref{eq:perf_func}) follows directly from (\ref%
{eq:a_per}), (\ref{eq:n_choice}) by noting that $\frac{1}{n^{2}}\leq \frac{1%
}{(6p\lambda )^{2}}$. The proof is complete.
\end{proof}

To prove the off-diagonal upper bound (\ref{eq:UE}), we need the following
two lemmas. We begin with the first one, Lemma \ref{lem:F-S} below, which
can be proved as in \cite[Lemma 3.7]{murugan2015heat}, see also \cite[Lemma
3.21]{carlen1987upper}.

\begin{lemma}
\label{lem:F-S}Let $w:(0,\infty )\rightarrow (0,\infty )$ be a
non-decreasing function and suppose that $u\in C^{1}([0,\infty );(0,\infty
)) $ satisfies that for all $t\geq 0$, 
\begin{equation}
u^{\prime }(t)\leq -b\frac{t^{p-2}}{w^{\theta }(t)}u^{1+\theta }(t)+Ku(t)
\label{eq:Ab_lemma_cond}
\end{equation}%
for some $b>0,$ $p>1$, $\theta >0$ and $K>0$. Then 
\begin{equation}
u(t)\leq \left( \frac{2p^{v}}{\theta b}\right) ^{1/\theta }t^{-(p-1)/\theta
}e^{Kp^{-\nu }t}w(t)  \label{eq:Ab_R}
\end{equation}%
for any $\nu \geq 1$.
\end{lemma}

We give the following second lemma that is of independent interest.

\begin{lemma}
\label{UEiter} Assume that condition \ref{eq:vol_l} holds. If the heat
kernel $p_{t}(x,y)$ satisfies%
\begin{equation}
p_{t}(x,y)\leq \frac{C}{t^{\alpha /\beta }}\exp (\,R_{0}^{-\beta }t\,)\exp
\left( -c\left( \frac{d(x,y)}{t^{1/\beta }}\right) ^{\frac{\beta }{\beta
^{\prime }-1}}\right)  \label{hkASS}
\end{equation}%
for $\mu $-almost all $x,y\in M$ and for all $t>0$, where $\beta ^{\prime
}>\beta >1$ and $C,c$ are independent of $R_{0}$, then it also satisfies \ref%
{eq:UE_loc} (that is, estimate (\ref{hkASS}) also holds with $\beta ^{\prime
}$ being replaced by $\beta $ and with some $C,c$ being independent of $%
R_{0} $).
\end{lemma}

\begin{remark}
\label{R:It}\textrm{\textrm{Lemma \ref{UEiter} is a self-improvement of the
heat kernel estimate, raising some power $\frac{\beta }{\beta ^{\prime }-1}$
to the power $\frac{\beta }{\beta -1}$, the best one possible. The smaller $%
\beta ^{\prime }$ is, the sharper (\ref{hkASS})$.$ } }
\end{remark}

The proof below is inspired by \cite[proof of Theorem 5.7, pp. 542-544]%
{grigor2014upper} wherein $\beta ^{\prime }=\beta +1$ and $R_{0}=\infty $.
We will see that $\beta ^{\prime }=2\beta +2$ in our application.

\begin{proof}
We claim that if $\beta ^{\prime }\geq \beta +1$, then (\ref{hkASS}) also
holds with $\beta ^{\prime }$ being replaced by $\beta ^{\prime }-1$. The
proof is quite long.

Let 
\begin{equation}
\theta :=\beta /\left( \beta ^{\prime }-1\right) .  \label{Y}
\end{equation}%
Clearly, $0<\theta \leq 1$. For $x\in M,t>0$, let 
\begin{equation}
r=2t^{1/\beta }/\delta  \label{rDef}
\end{equation}%
with $\delta >0$ to be chosen later. Set $B:=B(x,r)$, $B_{k}=kB$ ($k\geq 1$%
), $B_{0}=\emptyset $.

Let us show that for any $0<\varepsilon <1$, there exists some $\delta
=\delta (\varepsilon )>0$ such that%
\begin{equation}
P_{t}1_{B_{k}^{c}}\leq \exp (\,R_{0}^{-\beta }t\,)k^{\alpha }\varepsilon
^{k^{\theta }}\text{ in }\frac{1}{4}B  \label{Pt1B}
\end{equation}%
for any integer $k\geq 1$ and any $t>0$.

Indeed, if $B_{k}^{c}$ is empty, then (\ref{Pt1B}) is trivial since $%
P_{t}1_{B_{k}^{c}}=0$ in $M$. Assume that $B_{k}^{c}\neq \emptyset $. Using %
\ref{eq:vol_l} and (\ref{rDef}), we have from (\ref{hkASS}) that for $\mu $%
-almost all $y\in \frac{1}{4}B$ and all $t>0$,%
\begin{eqnarray}
P_{t}1_{B_{k}^{c}}(y) &\leq &\int_{M\setminus B(x,kr)}\frac{C}{t^{\alpha
/\beta }}\exp (\,R_{0}^{-\beta }t\,)\exp \left( -c\left( \frac{d(y,z)}{%
t^{1/\beta }}\right) ^{\frac{\beta }{\beta ^{\prime }-1}}\right) d\mu (z) 
\notag \\
&\leq &\exp (\,R_{0}^{-\beta }t\,)\int_{M\setminus B(x,kr)}\frac{C}{%
t^{\alpha /\beta }}\exp \left( -c^{\prime }\left( \frac{d(x,z)}{t^{1/\beta }}%
\right) ^{\frac{\beta }{\beta ^{\prime }-1}}\right) d\mu (z)  \notag \\
&\leq &\exp (\,R_{0}^{-\beta }t\,)C^{\prime }\int_{k/\delta }^{\infty
}s^{\alpha -1}\exp (-c^{\prime }s^{\theta })ds  \notag \\
&=&\exp (\,R_{0}^{-\beta }t\,)k^{\alpha }C^{\prime }\int_{1/\delta }^{\infty
}s^{\alpha -1}\exp (-c^{\prime }\left( ks\right) ^{\theta })ds  \label{Pt2}
\end{eqnarray}%
for any $k\geq 1$ (see \cite[formula (3.7)]{grigor2003heat}), where $%
C^{\prime },c^{\prime }$ are independent of $R_{0}$.

For any $0<\varepsilon <1,$ choose $\delta >0$ to be so small that both of
what follows are satisfied:%
\begin{eqnarray*}
C^{\prime }\int_{1/\delta }^{\infty }s^{\alpha -1}\exp \left( -\frac{%
c^{\prime }}{2}s^{\theta }\right) ds &\leq &1, \\
\exp \left( -\frac{c^{\prime }}{2}\delta ^{-\theta }\right) &\leq
&\varepsilon .
\end{eqnarray*}%
From this, we have%
\begin{eqnarray*}
C^{\prime }\int_{1/\delta }^{\infty }s^{\alpha -1}\exp (-c^{\prime }\left(
ks\right) ^{\theta })ds &=&C^{\prime }\int_{1/\delta }^{\infty }\exp \left( -%
\frac{c^{\prime }}{2}k^{\theta }s^{\theta }\right) \cdot s^{\alpha -1}\exp
\left( -\frac{c^{\prime }}{2}k^{\theta }s^{\theta }\right) ds \\
&\leq &\exp \left( -\frac{c^{\prime }}{2}k^{\theta }\delta ^{-\theta
}\right) \cdot C^{\prime }\int_{1/\delta }^{\infty }s^{\alpha -1}\exp \left(
-\frac{c^{\prime }}{2}s^{\theta }\right) ds \\
&\leq &\left\{ \exp \left( -\frac{c^{\prime }}{2}\delta ^{-\theta }\right)
\right\} ^{k^{\theta }}\leq \varepsilon ^{k^{\theta }},
\end{eqnarray*}%
Therefore, by (\ref{Pt2}),%
\begin{equation*}
P_{t}1_{B_{k}^{c}}(y)\leq \exp (\,R_{0}^{-\beta }t\,)k^{\alpha }C^{\prime
}\int_{1/\delta }^{\infty }s^{\alpha -1}\exp (-c^{\prime }\left( ks\right)
^{\theta })ds\leq \exp (\,R_{0}^{-\beta }t\,)k^{\alpha }\varepsilon
^{k^{\theta }},
\end{equation*}%
thus proving (\ref{Pt1B}).

Define the function $E_{t,x}$ by 
\begin{equation}
E_{t,x}(\cdot )=\exp \left( a\left( \frac{d(x,\cdot )}{t^{1/\beta }}\right)
^{\theta }\right)  \label{Et,x}
\end{equation}%
for some constant $a>0$ to be determined later. Let us show that for all $%
t>0 $ and all $x\in M$,%
\begin{equation}
P_{t}E_{t,x}\leq A_{1}\exp (\,R_{0}^{-\beta }t\,)\text{ a.a. in }B(x,r/4),
\label{PtA1}
\end{equation}%
where $A_{1}$ is some constant depending on $\varepsilon ,\delta $ only.

Indeed, by (\ref{Et,x}) and (\ref{Pt1B}), (\ref{rDef}), we have that in $%
\frac{1}{4}B$, 
\begin{align*}
P_{t}E_{t,x}& =\overset{\infty }{\underset{k=0}{\sum }}P_{t}\left(
1_{B_{k+1}\setminus B_{k}}E_{t,x}\right) \leq \overset{\infty }{\underset{k=0%
}{\sum }}\left\Vert E_{t,x}\right\Vert _{L^{\infty
}(B_{k+1})}P_{t}1_{B_{k+1}\setminus B_{k}} \\
& \leq \overset{\infty }{\underset{k=0}{\sum }}\exp \left( a\left( \frac{%
(k+1)r}{t^{1/\beta }}\right) ^{\theta }\right) \cdot P_{t}1_{B_{k}^{c}} \\
& \leq \overset{\infty }{\underset{k=0}{\sum }}\exp \left( a2^{\theta
}\left( \frac{k+1}{\delta }\right) ^{\theta }\right) \cdot \exp
(\,R_{0}^{-\beta }t\,)k^{\alpha }\varepsilon ^{k^{\theta }}.
\end{align*}%
Choose $a<\frac{1}{3}\left( \delta /2\right) ^{\theta }\log (1/\varepsilon )$
such that this series converges, proving (\ref{PtA1}).

Let us show that for\ all $t>0$ and all $x\in M$, 
\begin{equation}
P_{t}E_{t,x}\leq A_{2}\exp (\,R_{0}^{-\beta }t\,)E_{t,x}\text{ in }M,
\label{PtE}
\end{equation}%
for some constant $A_{2}=A_{2}(\varepsilon ,\delta )$.

Indeed, using the elementary inequality that $(a+b)^{\theta }\leq a^{\theta
}+b^{\theta }$ for any $a,b\geq 0$ and any $0\leq \theta \leq 1$, we have
that for any $x,y,z\in M$ and $t>0$,%
\begin{eqnarray*}
E_{t,x}(y) &=&\exp \left( a\left( \frac{d(x,y)}{t^{1/\beta }}\right)
^{\theta }\right) \\
&\leq &\exp \left( a\left( \frac{d(x,z)}{t^{1/\beta }}\right) ^{\theta
}\right) \exp \left( a\left( \frac{d(z,y)}{t^{1/\beta }}\right) ^{\theta
}\right) =E_{t,x}(z)E_{t,z}(y),
\end{eqnarray*}%
that is, $E_{t,x}\leq E_{t,x}(z)E_{t,z}$, and thus%
\begin{equation}
P_{t}E_{t,x}\leq E_{t,x}(z)P_{t}E_{t,z}.  \label{Ptl}
\end{equation}%
Note that by (\ref{PtA1})%
\begin{equation}
P_{t}E_{t,z}\leq A_{1}\exp (\,R_{0}^{-\beta }t\,)\text{ a.a. in }B(z,r/4)
\label{PtA}
\end{equation}%
for all $t>0$. For all $y\in B(z,r/4)$, by (\ref{rDef}),%
\begin{equation*}
E_{t,y}(z)\leq \exp \left( a\left( \frac{r}{4t^{1/\beta }}\right) ^{\theta
}\right) =\exp \left( a\left( 2\delta \right) ^{-\theta }\right) :=A_{3},
\end{equation*}%
and hence,%
\begin{equation*}
E_{t,x}(z)\leq E_{t,x}(y)E_{t,y}(z)\leq A_{3}E_{t,x}(y).
\end{equation*}%
It follows from (\ref{Ptl}), (\ref{PtA}) that%
\begin{equation*}
P_{t}E_{t,x}\leq A_{1}A_{3}\exp (\,R_{0}^{-\beta }t\,)E_{t,x}\text{ a.a. in }%
B(z,r/4).
\end{equation*}%
Since the point $z$ is arbitrary, we cover $M$ by a countable sequence of
balls like $B(z,r)$, and obtain that (\ref{PtE}) is true with $%
A_{2}=A_{1}A_{3}$.

Let us show that for all $t>0,x\in M$, and for any integer $k\geq 1$, 
\begin{equation}
P_{kt}E_{t,x}\leq \exp (\,kR_{0}^{-\beta }t\,)A_{2}^{k}\text{ a.a. in }%
B(x,r/4)\text{ with }r=2t^{1/\beta }/\delta .  \label{PtB}
\end{equation}%
Indeed, by (\ref{PtE})%
\begin{equation*}
P_{kt}E_{t,x}=P_{(k-1)t}P_{t}E_{t,x}\leq \left\{ A_{2}\exp (\,R_{0}^{-\beta
}t\,)\right\} P_{(k-1)t}E_{t,x}\leq \ldots \leq A_{2}^{k-1}\exp
(\,(k-1)R_{0}^{-\beta }t\,)P_{t}E_{t,x},
\end{equation*}%
which together with (\ref{PtA1}) gives (\ref{PtB}), where we have used $%
A_{2}\geq A_{1}$.

Fix $B_{R}:=B(x_{0},R)$ for any $R>0$ and any $x_{0}\in M$. We show that%
\begin{equation}
P_{t}1_{B_{R}^{c}}\leq A_{0}\exp (\,R_{0}^{-\beta }t\,)\exp \left( a^{\prime
}\lambda t-a\left( R\lambda ^{1/\beta }\right) ^{\theta }\right) \text{ in }%
\frac{1}{2}B_{R}  \label{Cov}
\end{equation}%
for all $t>0$ and all $\lambda >0$, where constants $A_{0},a^{\prime }$
depend on $\varepsilon ,\delta $ only.

Indeed, assume that $B_{R}^{c}\neq \emptyset $; otherwise (\ref{Cov}) is
trivial. Observe that for any $x\in \frac{1}{2}B_{R}$,%
\begin{equation*}
P_{t}1_{B_{R}^{c}}\leq P_{t}1_{B(x,R/2)^{c}}.
\end{equation*}%
It suffices to show that for all $x\in \frac{1}{2}B_{R}$ and all $t>0,$ 
\begin{equation}
P_{t}1_{B(x,R/2)^{c}}\leq A_{0}\exp (\,R_{0}^{-\beta }t\,)\exp \left(
a^{\prime }\lambda t-a\left( R\lambda ^{1/\beta }\right) ^{\theta }\right)
\label{46}
\end{equation}%
in a (small) ball containing $x$. Then covering $\frac{1}{2}B_{R}$ by a
countable family of such balls, we obtain (\ref{Cov}).

To see this, replacing $t$ by $t/k$ in (\ref{PtB}), we have that for all $%
t>0,x\in M$ and any $k\geq 1$, 
\begin{equation*}
P_{t}E_{t/k,x}\leq \exp (\,R_{0}^{-\beta }t\,)A_{2}^{k}\text{ in }B(x,r_{k}),
\end{equation*}%
where $r_{k}=\left( t/k\right) ^{1/\beta }/(2\delta )$. Since%
\begin{equation*}
E_{t/k,x}\geq \exp \left( a\left( \frac{R}{\left( t/k\right) ^{1/\beta }}%
\right) ^{\theta }\right) \text{ in }B(x,R)^{c},
\end{equation*}%
we have that%
\begin{equation*}
1_{B(x,R)^{c}}\leq \exp \left( -a\left( \frac{R}{\left( t/k\right) ^{1/\beta
}}\right) ^{\theta }\right) E_{t/k,x}\text{ in }M.
\end{equation*}%
It follows that for all $t>0,x\in M$%
\begin{equation*}
P_{t}1_{B(x,R)^{c}}\leq \exp \left( -a\left( \frac{R}{\left( t/k\right)
^{1/\beta }}\right) ^{\theta }\right) P_{t}E_{t/k,x}\leq \exp
(\,R_{0}^{-\beta }t\,)\exp \left( a^{\prime }k-a\left( \frac{R}{\left(
t/k\right) ^{1/\beta }}\right) ^{\theta }\right)
\end{equation*}%
in $B(x,r_{k})$, where $a^{\prime }=\log A_{2}$. Given any $\lambda >0$ and
any $t>0$, we can choose an integer $k\geq 1$ such that%
\begin{equation*}
\lambda t\leq k<\lambda t+1.
\end{equation*}%
With such choice of $k$, we conclude that for all $t>0,x\in M$ and all $%
\lambda >0$,%
\begin{equation*}
P_{t}1_{B(x,R)^{c}}\leq \exp (\,R_{0}^{-\beta }t\,)\exp \left( a^{\prime
}\left( \lambda t+1\right) -a\left( \frac{R}{\left( 1/\lambda \right)
^{1/\beta }}\right) ^{\theta }\right)
\end{equation*}%
in $B(x,r_{k})$, which finishes the proof of (\ref{46}), and also of (\ref%
{Cov}).

Choosing $\lambda $ in (\ref{Cov}) such that $a^{\prime }\lambda t=a\left(
R\lambda ^{1/\beta }\right) ^{\theta }/2$, that is, 
\begin{equation*}
\lambda =\left( \frac{aR^{\theta }}{2a^{\prime }t}\right) ^{\frac{\beta }{%
\beta -\theta }},
\end{equation*}%
we conclude that for all $t>0$,%
\begin{equation}
P_{t}1_{B_{R}^{c}}\leq A_{0}\exp (\,R_{0}^{-\beta }t\,)\exp \left(
-a^{\prime }\lambda t\right) =A_{0}\exp (\,R_{0}^{-\beta }t\,)\exp \left(
-c\left( \frac{R}{t^{1/\beta }}\right) ^{\frac{\beta }{\beta /\theta -1}%
}\right)  \label{47}
\end{equation}%
in $B(x_{0},R/2)$, for some universal constant $c>0$ (also independent of $%
R_{0}$).

For two distinct points $x_{0},y_{0}\in M$, let $R=$ $d(x_{0},y_{0})/2$. By
the semigroup property, 
\begin{eqnarray}
p_{2t}\left( x,y\right) &=&\int_{M}p_{t}(x,z)p_{t}(z,y)d\mu (z)  \notag \\
&\leq &\left\{ \int_{B(x_{0},R)^{c}}+\int_{B(y_{0},R)^{c}}\right\}
p_{t}(x,z)p_{t}(z,y)d\mu (z)=:I_{1}(x,y)+I_{2}(x,y).  \label{I}
\end{eqnarray}%
Using (\ref{eq:DUE}) and (\ref{47}), we have that for all $t>0$ and $\mu $%
-almost all $x\in B(x_{0},R/2),y\in M$, 
\begin{eqnarray}
I_{1}(x,y) &=&\int_{B(x_{0},R)^{c}}p_{t}(x,z)p_{t}(z,y)d\mu (z)  \notag \\
&\leq &\underset{z\in B(x_{0},R)^{c}}{\sup }p_{t}(z,y)\cdot
\int_{B(x_{0},R)^{c}}p_{t}(x,z)d\mu (z)  \notag \\
&\leq &\frac{C}{t^{\alpha /\beta }}\exp (\,2R_{0}^{-\beta }t\,)\exp \left(
-c\left( \frac{R}{t^{1/\beta }}\right) ^{\frac{\beta }{\beta /\theta -1}%
}\right)  \notag \\
&=&\frac{C}{t^{\alpha /\beta }}\exp (\,2R_{0}^{-\beta }t\,)\exp \left(
-c\left( \frac{d(x_{0},y_{0})}{2t^{1/\beta }}\right) ^{\frac{\beta }{\beta
/\theta -1}}\right) .  \label{I1}
\end{eqnarray}%
Similarly, for all $t>0$ and $\mu $-almost all $y\in B(y_{0},R/2),x\in M$,%
\begin{equation}
I_{2}(x,y)\leq \frac{C}{t^{\alpha /\beta }}\exp (\,2R_{0}^{-\beta }t\,)\exp
\left( -c\left( \frac{d(x_{0},y_{0})}{2t^{1/\beta }}\right) ^{\frac{\beta }{%
\beta /\theta -1}}\right) .  \label{I_2}
\end{equation}%
Plugging (\ref{I1}), (\ref{I_2}) into (\ref{I}) and then renaming $2t$ by $t$%
, we obtain that (\ref{hkASS}) holds with $\beta ^{\prime }$ being replaced
by $\beta ^{\prime }-1=\beta /\theta $ (cf. \cite[pp. 183-184]%
{grigor14heatkernelMMS}), thus proving our claim.

Finally, repeat our claim $k$\ times until the integer $k$ satisfies%
\begin{equation*}
\beta <\beta ^{\prime }-k\leq \beta +1,
\end{equation*}%
that is, $1\leq \frac{\beta }{\beta ^{\prime }-k-1}<\frac{\beta }{\beta -1}$%
. Then (\ref{hkASS}) holds with $\beta ^{\prime }$ being replaced by $\beta
^{\prime }-k$, which also implies that (\ref{hkASS}) holds with $\beta
^{\prime }=\beta +1$ by reducing the value of $\frac{\beta }{\beta ^{\prime
}-k-1}$ to $1$. From this, we repeat the claim one more time (where $\theta
=1$), and obtain \ref{eq:UE_loc}, as desired.
\end{proof}

We are now in a position to prove Theorem \ref{thm:Mr2}.

\begin{proof}[Proof of Theorem \protect\ref{thm:Mr2}]
Fix $x_{0}\in M$. Let $f\in \mathcal{F}\cap L^{\infty }$ be nonnegative with 
$\left\Vert f\right\Vert _{2}=1$. For $0<r<R_{0}/2$, set 
\begin{equation}
\rho :=\eta r,  \label{eq:rho_setup}
\end{equation}%
where $0<\eta <1$ will be specified later. For any integer $k\geq 0$, set $%
p_{k}=2^{k}$ and%
\begin{equation}
\psi _{k}:=\lambda \phi _{p_{k},\lambda }  \label{eq:Psi_choice}
\end{equation}%
where $\phi _{p_{k},\lambda }\in $\textrm{cutoff}$(B(x_{0},r),B(x_{0},2r))$
is given by (\ref{eq:R}) with $p=p_{k}$ and with $\lambda \geq \eta ^{-1}$
to be chosen later. Clearly, for $\mu $-almost all $x\in B(x_{0},r),y\in
M\setminus B(x_{0},2r)$ 
\begin{equation}
\psi _{k}(y)-\psi _{k}(x)=\lambda \cdot 0-\lambda \cdot 1=-\lambda .
\label{eq:Psi_prop}
\end{equation}

Let $(\mathcal{E}_{\rho },\mathcal{F}_{\rho })$ be the truncated Dirichlet
form given by (\ref{eq:DF_q}). Denote by $q_{t}^{(\rho )}(x,y)$, $\left\{
Q_{t}\right\} _{t\geq 0}$ the heat kernel and heat semigroup associated with 
$(\mathcal{E}_{\rho },\mathcal{F}_{\rho })$ respectively. We define the
\textquotedblleft perturbed semigroup\textquotedblright\ by 
\begin{equation*}
Q_{t}^{\psi _{k}}f=e^{\psi _{k}}\left( Q_{t}\left( e^{-\psi _{k}}f\right)
\right) .
\end{equation*}%
For simplicity set for any integer $k\geq 0$ 
\begin{equation}
f_{t,k}:=Q_{t}^{\psi _{k}}f.  \label{eq:p_sgrp}
\end{equation}%
Clearly, the function $f_{t,k}\in \mathcal{F}\cap L^{\infty }\subset 
\mathcal{F}_{\rho }$.

By applying (\ref{eq:R}) again with $p=p_{k},R=r,\phi =\phi _{p_{k},\lambda
} $ and but with $f$ being replaced by $f_{t,k}$ this time, and by setting 
\begin{equation}
K_{0}:=TC_{0}\lambda ^{2\beta +2}  \label{eq:K_def}
\end{equation}%
with $T$ given by (\ref{tt}), we obtain that for any $k\geq 0$,%
\begin{equation*}
\mathcal{E}_{\rho }(e^{-\psi _{k}}f_{t,k},e^{\psi
_{k}}f_{t,k}^{2p_{k}-1})\geq \frac{1}{4p_{k}}\mathcal{E}%
(f_{t,k}^{p_{k}})-K_{0}p_{k}^{2\beta +1}\int f_{t,k}^{2p_{k}}d\mu .
\end{equation*}%
From this, we derive that%
\begin{align}
\frac{d}{dt}\left\Vert f_{t,k}\right\Vert _{2p_{k}}^{2p_{k}}& =-2p_{k}%
\mathcal{E}_{\rho }(e^{-\psi _{k}}f_{t,k},e^{\psi _{k}}f_{t,k}^{2p_{k}-1}) 
\notag \\
& \leq -2p_{k}\left\{ \frac{1}{4p_{k}}\mathcal{E}%
(f_{t,k}^{p_{k}})-K_{0}p_{k}^{2\beta +1}\left\Vert f_{t,k}\right\Vert
_{2p_{k}}^{2p_{k}}\right\}  \notag \\
& =-\frac{1}{2}\mathcal{E}(f_{t,k}^{p_{k}})+2K_{0}p_{k}^{2\beta
+2}\left\Vert f_{t,k}\right\Vert _{2p_{k}}^{2p_{k}}.  \label{eq:st_k1}
\end{align}%
In particular, for $k=0$ ($p_{0}=1$), 
\begin{equation*}
\frac{d}{dt}\left\Vert f_{t,0}\right\Vert _{2}^{2}\leq 2K_{0}\left\Vert
f_{t,0}\right\Vert _{2}^{2},
\end{equation*}%
which gives that, using $\left\Vert f_{0,0}\right\Vert _{2}=\left\Vert
f\right\Vert _{2}=1$, 
\begin{equation}
\left\Vert f_{t,0}\right\Vert _{2}=\left\Vert f_{t,0}\right\Vert
_{p_{1}}\leq e^{K_{0}t}\left\Vert f\right\Vert _{2}=e^{K_{0}t}.
\label{eq:st_k2}
\end{equation}%
Since condition (\ref{eq:DUE}) implies the Nash inequality (cf. \cite[%
Theorem 2.1]{carlen1987upper}):%
\begin{equation*}
\left\Vert u\right\Vert _{2}^{2(1+\frac{\beta }{\alpha })}\leq C_{N}\left( 
\mathcal{E}(u)+R_{0}^{-\beta }\left\Vert u\right\Vert _{2}^{2}\right)
\left\Vert u\right\Vert _{1}^{2\beta /\alpha }
\end{equation*}%
for all $u\in \mathcal{F}\cap L^{1}$, we apply this inequality to function $%
f_{t,k}^{p_{k}}\in \mathcal{F}\cap L^{1}$ with $k\geq 1$%
\begin{equation*}
\mathcal{E}(f_{t,k}^{p_{k}})\geq \frac{1}{C_{N}}\left\Vert
f_{t,k}\right\Vert _{2p_{k}}^{2p_{k}(1+\frac{\beta }{\alpha })}\cdot
\left\Vert f_{t,k}\right\Vert _{p_{k}}^{-2p_{k}\beta /\alpha }-R_{0}^{-\beta
}\left\Vert f_{t,k}\right\Vert _{2p_{k}}^{2p_{k}}.
\end{equation*}%
Plugging this into (\ref{eq:st_k1}), we have%
\begin{eqnarray*}
2p_{k}\left\Vert f_{t,k}\right\Vert _{2p_{k}}^{2p_{k}-1}\frac{d}{dt}%
\left\Vert f_{t,k}\right\Vert _{2p_{k}} &=&\frac{d}{dt}\left\Vert
f_{t,k}\right\Vert _{2p_{k}}^{2p_{k}} \\
&\leq &-\frac{1}{2C_{N}}\left\Vert f_{t,k}\right\Vert _{2p_{k}}^{2p_{k}(1+%
\frac{\beta }{\alpha })}\cdot \left\Vert f_{t,k}\right\Vert
_{p_{k}}^{-2p_{k}\beta /\alpha }+\left( \frac{1}{2}R_{0}^{-\beta
}+2K_{0}p_{k}^{2\beta +2}\right) \left\Vert f_{t,k}\right\Vert
_{2p_{k}}^{2p_{k}},
\end{eqnarray*}%
which implies that%
\begin{equation}
\frac{d}{dt}\left\Vert f_{t,k}\right\Vert _{2p_{k}}\leq -\frac{1}{4C_{N}p_{k}%
}\left\Vert f_{t,k}\right\Vert _{2p_{k}}^{1+\frac{2p_{k}\beta }{\alpha }%
}\left\Vert f_{t,k}\right\Vert _{p_{k}}^{-\frac{2p_{k}\beta }{\alpha }%
}+b_{k}p_{k}^{2\beta +1}\left\Vert f_{t,k}\right\Vert _{2p_{k}}
\label{eq:st_k3}
\end{equation}%
for all $k\geq 1$, where%
\begin{equation}
b_{k}:=K_{0}+\left( 4R_{0}^{\beta }p_{k}^{2\beta +2}\right) ^{-1}\leq K_{0}+%
\frac{1}{4}R_{0}^{-\beta }\text{ for any }k\geq 0\text{.}  \label{bk}
\end{equation}

On the other hand, we claim that for any $k\geq 0$,%
\begin{equation}
\exp (-3/p_{k})f_{t,k+1}\leq f_{t,k}\leq \exp (3/p_{k})f_{t,k+1}.
\label{eq:compare}
\end{equation}%
Indeed, observe from (\ref{eq:Psi_choice}), (\ref{eq:perf_func}) and $%
p_{k+1}=2p_{k}$,%
\begin{eqnarray}
\left\Vert \psi _{k+1}-\psi _{k}\right\Vert _{\infty } &=&\lambda \left\Vert
\phi _{p_{k+1},\lambda }-\phi _{p_{k},\lambda }\right\Vert _{\infty }  \notag
\\
&\leq &\lambda \left\Vert \phi _{p_{k+1},\lambda }-\Phi \right\Vert _{\infty
}+\lambda \left\Vert \phi _{p_{k},\lambda }-\Phi \right\Vert _{\infty } 
\notag \\
&\leq &\lambda \left( \frac{1}{2\lambda p_{k}}+\frac{1}{\lambda p_{k}}%
\right) =\frac{3}{2p_{k}}.  \label{eq:psi_k_es}
\end{eqnarray}%
From this and using the Markovian property of $\left\{ Q_{t}\right\} _{t\geq
0}$, we have%
\begin{equation*}
f_{t,k}=e^{\psi _{k}}\left( Q_{t}\left( e^{-\psi _{k}}f\right) \right) \leq
e^{\psi _{k+1}+\frac{3}{2p_{k}}}\left( Q_{t}\left( e^{-\psi _{k+1}+\frac{3}{%
2p_{k}}}f\right) \right) =e^{3/p_{k}}f_{t,k+1}.
\end{equation*}%
Similarly,%
\begin{equation*}
f_{t,k+1}\leq e^{3/p_{k}}f_{t,k}.
\end{equation*}%
Thus, we obtain (\ref{eq:compare}), proving our claim.

Therefore, we conclude from (\ref{eq:st_k3}) that, using the fact that $%
f_{t,k}\leq e^{6/p_{k}}f_{t,k-1}$ by (\ref{eq:compare}), 
\begin{equation}
\frac{d}{dt}\left\Vert f_{t,k}\right\Vert _{2p_{k}}\leq -\frac{1}{%
C_{N}^{\prime }p_{k}}\left\Vert f_{t,k}\right\Vert _{2p_{k}}^{1+\frac{%
2p_{k}\beta }{\alpha }}\left\Vert f_{t,k-1}\right\Vert _{p_{k}}^{-\frac{%
2p_{k}\beta }{\alpha }}+b_{k}p_{k}^{2\beta +1}\left\Vert f_{t,k}\right\Vert
_{2p_{k}}  \label{eq:st_k4}
\end{equation}%
for all $k\geq 1$, where $C_{N}^{\prime }=4C_{N}\exp (12\beta /\alpha )$.

Define $u_{k}(t):=\left\Vert f_{t,k-1}\right\Vert _{p_{k}}$ and 
\begin{equation}
w_{k}(t):=\underset{s\in (0,t]}{\sup }\left\{ s^{\alpha (p_{k}-2)/(2\beta
p_{k})}u_{k}(s)\right\} .  \label{eq:w_Def}
\end{equation}%
Then by (\ref{eq:st_k2})%
\begin{equation}
w_{1}(t)=\underset{s\in (0,t]}{\sup }\left\{ u_{1}(s)\right\} =\underset{%
s\in (0,t]}{\sup }\left\{ \left\Vert f_{t,0}\right\Vert _{2}\right\} \leq
e^{K_{0}t}.  \label{eq:st0_1}
\end{equation}

On the other hand, we have from (\ref{eq:st_k4}) that, using $u_{k}^{-2\beta
p_{k}/\alpha }(t)\geq t^{p_{k}-2}w_{k}^{-2\beta p_{k}/\alpha }(t)$, 
\begin{eqnarray*}
u_{k+1}^{\prime }(t) &=&\frac{d}{dt}\left\Vert f_{t,k}\right\Vert
_{2p_{k}}\leq -\frac{1}{C_{N}^{\prime }p_{k}}u_{k+1}^{1+2\beta p_{k}/\alpha
}(t)u_{k}^{-2\beta p_{k}/\alpha }(t)+b_{k}p_{k}^{2\beta +1}u_{k+1}(t) \\
&\leq &-\frac{1}{C_{N}^{\prime }p_{k}}\cdot \frac{t^{p_{k}-2}}{w_{k}^{2\beta
p_{k}/\alpha }(t)}u_{k+1}^{1+2\beta p_{k}/\alpha }+b_{k}p_{k}^{2\beta
+1}u_{k+1}(t)
\end{eqnarray*}%
for $k\geq 1$. Then the condition (\ref{eq:Ab_lemma_cond}) is satisfied with 
$u(t)=u_{k+1}(t)$, $b=1/\left( C_{N}^{\prime }p_{k}\right) $, $p=p_{k}\geq 2$%
, $\theta =2\beta p_{k}/\alpha $, $w=w_{k}$ and $K=b_{k}p_{k}^{2\beta +1}$.
Thus, applying Lemma \ref{lem:F-S} with $\nu =2\beta +2$, we obtain%
\begin{equation*}
u_{k+1}(t)\leq \left( C_{N}^{\prime }\alpha p_{k}^{2\beta +2}/\beta \right)
^{\alpha /\left( 2\beta p_{k}\right) }t^{-\alpha \left( p_{k}-1\right)
/\left( 2\beta p_{k}\right) }e^{b_{k}p_{k}^{-1}t}w_{k}(t),
\end{equation*}%
that is,%
\begin{equation}
t^{\alpha \left( p_{k+1}-2\right) /\left( 2\beta p_{k+1}\right)
}u_{k+1}(t)\leq \left( C_{N}^{\prime }\alpha p_{k}^{2\beta +2}/\beta \right)
^{\alpha /\left( 2\beta p_{k}\right) }e^{b_{k}p_{k}^{-1}t}w_{k}(t),
\label{eq:w_k_k+1}
\end{equation}%
for all $t>0$. From this, we derive that%
\begin{equation*}
w_{k+1}(t)=\underset{s\in (0,t]}{\sup }\left\{ s^{\alpha \left(
p_{k+1}-2\right) /\left( 2\beta p_{k+1}\right) }u_{k+1}(s)\right\} \leq
\left( C_{N}^{\prime }\alpha p_{k}^{2\beta +2}/\beta \right) ^{\alpha
/\left( 2\beta p_{k}\right) }e^{b_{k}p_{k}^{-1}t}w_{k}(t),
\end{equation*}%
which gives that, using (\ref{bk}), 
\begin{eqnarray}
w_{k+1}(t)/w_{k}(t) &\leq &\left( C_{N}^{\prime }\alpha p_{k}^{2\beta
+2}/\beta \right) ^{\alpha /\left( 2\beta p_{k}\right) }e^{b_{k}tp_{k}^{-1}}
\notag \\
&=&\left( 2^{k\left( 2\beta +2\right) }\cdot C_{N}^{\prime }\alpha /\beta
\right) ^{\alpha /\left( \beta 2^{k+1}\right) }e^{b_{k}t2^{-k}}  \notag \\
&=&\left\{ \left( C_{N}^{\prime }\alpha /\beta \right) ^{\alpha /\left(
2\beta \right) }e^{b_{k}t}\cdot \left( 2^{\alpha (\beta +1)/\beta }\right)
^{k}\right\} ^{2^{-k}}\leq (Da^{k})^{2^{-k}},  \notag
\end{eqnarray}%
where $D:=\left( C_{N}^{\prime }\alpha /\beta \right) ^{\alpha /\left(
2\beta \right) }e^{(K_{0}+\frac{1}{4}R_{0}^{-\beta })t}$ and $a:=2^{\alpha
(\beta +1)/\beta }$. This implies by iteration and using (\ref{eq:st0_1})
that for any $k\geq 1$, 
\begin{eqnarray}
w_{k+1}(t) &\leq &(Da^{k})^{1/2^{k}}w_{k}(t)  \notag \\
&\leq &(Da^{k})^{1/2^{k}}\left\{ (Da^{k-1})^{1/2^{k-1}}w_{k-1}(t)\right\}
\leq \ldots  \notag \\
&\leq &D^{\frac{1}{2^{k}}+\frac{1}{2^{k-1}}+\cdots +\frac{1}{2}}a^{\frac{k}{%
2^{k}}+\frac{k-1}{2^{k-1}}+\cdots +\frac{1}{2}}w_{1}(t)  \notag \\
&\leq &Da^{2}w_{1}(t)\leq Da^{2}e^{K_{0}t}=C_{7}\exp (\,2K_{0}t+\frac{1}{4}%
R_{0}^{-\beta }t\,),  \label{eq:w_it}
\end{eqnarray}%
where $C_{7}=\left( C_{N}^{\prime }\alpha /\beta \right) ^{\alpha /\left(
2\beta \right) }2^{2\alpha (\beta +1)/\beta }$. Thus, we have from (\ref%
{eq:w_Def}), (\ref{eq:w_it}) that for any $k\geq 1$, $t>0$%
\begin{equation}
t^{\alpha (p_{k+1}-2)/(2\beta p_{k+1})}\left\Vert Q_{t}^{\psi
_{k}}f\right\Vert _{2p_{k}}\leq w_{k+1}(t)\leq C_{7}\exp (\,2K_{0}t+\frac{1}{%
4}R_{0}^{-\beta }t\,)  \label{eq:it_r1}
\end{equation}%
for any $0\leq f\in \mathcal{F}\cap L^{\infty }$ with $\left\Vert
f\right\Vert _{2}=1$.

Since $\psi _{k}$ is a Cauchy sequence in $L^{\infty }$ by (\ref{eq:psi_k_es}%
), the sequence $\left\{ \psi _{k}\right\} $ converges uniformly to $\psi
_{\infty }$ as $k\rightarrow \infty $ with 
\begin{equation*}
\psi _{\infty }:=\lambda \psi \in L^{\infty }
\end{equation*}%
by using (\ref{eq:perf_func}), where $\psi (y)=\left( \frac{2r-d(x_{0},y)}{r}%
\right) _{+}\wedge 1$ for $y\in M$. Clearly,%
\begin{equation}
\psi _{\infty }(y)-\psi _{\infty }(x)=-\lambda  \label{eq:Psip}
\end{equation}%
for any $x\in B(x_{0},r)$ and any $y\in M\setminus B(x_{0},2r)$. Set 
\begin{equation*}
f_{t,\infty }:=e^{\psi _{\infty }}\left( Q_{t}\left( e^{-\psi _{\infty
}}f\right) \right) .
\end{equation*}%
The sequence $\left\{ f_{t,k}\right\} _{k\geq 1}$ converges uniformly to $%
f_{t,\infty }$ as $k\rightarrow \infty $, and thus%
\begin{equation*}
\left\Vert Q_{t}^{\psi _{k}}f\right\Vert _{p_{k}}=\left\Vert
f_{t,k}\right\Vert _{p_{k}}\rightarrow \left\Vert Q_{t}^{\psi _{\infty
}}f\right\Vert _{\infty }.
\end{equation*}%
Therefore, letting $k\rightarrow \infty $ in (\ref{eq:it_r1}), we obtain
that 
\begin{equation*}
\left\Vert Q_{t}^{\psi _{\infty }}f\right\Vert _{\infty }\leq \frac{C_{7}}{%
t^{\alpha /\left( 2\beta \right) }}\exp (\,2K_{0}t+\frac{1}{4}R_{0}^{-\beta
}t\,).
\end{equation*}%
for any $0\leq f\in \mathcal{F}\cap L^{\infty }$ with $\left\Vert
f\right\Vert _{2}=1$, that is,%
\begin{equation*}
\left\Vert Q_{t}^{\psi _{\infty }}\right\Vert _{2\rightarrow \infty
}:=\sup_{\left\Vert f\right\Vert _{2}=1}\left\Vert Q_{t}^{\psi _{\infty
}}f\right\Vert _{\infty }\leq \frac{C_{7}}{t^{\alpha /\left( 2\beta \right) }%
}\exp (\,2K_{0}t+\frac{1}{4}R_{0}^{-\beta }t\,).
\end{equation*}%
This inequality is also true for $-\psi _{\infty }$ by Remark \ref{R11} and
repeating the above procedure. Since $Q_{t}^{-\psi _{\infty }}$ is the
adjoint of operator $Q_{t}^{\psi _{\infty }}$, we see that 
\begin{equation*}
\left\Vert Q_{t}^{\psi _{\infty }}\right\Vert _{1\rightarrow
2}:=\sup_{\left\Vert f\right\Vert _{1}=1}\left\Vert Q_{t}^{\psi _{\infty
}}f\right\Vert _{2}=\left\Vert Q_{t}^{-\psi _{\infty }}\right\Vert
_{2\rightarrow \infty }\leq \frac{C_{7}}{t^{\alpha /\left( 2\beta \right) }}%
\exp (\,2K_{0}t+\frac{1}{4}R_{0}^{-\beta }t\,),
\end{equation*}%
and thus, 
\begin{equation*}
\left\Vert Q_{t}^{\psi _{\infty }}\right\Vert _{1\rightarrow \infty }\leq
\left\Vert Q_{t/2}^{\psi _{\infty }}\right\Vert _{1\rightarrow 2}\left\Vert
Q_{t/2}^{\psi _{\infty }}\right\Vert _{2\rightarrow \infty }\leq \frac{C_{8}%
}{t^{\alpha /\beta }}\exp (\,2K_{0}t+\frac{1}{4}R_{0}^{-\beta }t\,)
\end{equation*}%
where $C_{8}=2^{\alpha /\beta }\left( C_{7}\right) ^{2}$. From this and
using (\ref{eq:Psip}), (\ref{eq:K_def}), 
\begin{align}
q_{t}^{(\rho )}(x,y)& \leq \frac{C_{8}}{t^{\alpha /\beta }}\exp \left(
2K_{0}t+\frac{1}{4}R_{0}^{-\beta }t+\psi _{\infty }(y)-\psi _{\infty
}(x)\right)  \notag \\
& =\frac{C_{8}}{t^{\alpha /\beta }}\exp (\,\frac{1}{4}R_{0}^{-\beta
}t\,)\exp \left( 2C_{0}\lambda ^{2\beta +2}Tt-\lambda \right)  \label{44}
\end{align}%
for all $t>0$, $r\in (0,R_{0}/2)$, $\mu $-almost all $x\in B(x_{0},r),y\in
M\setminus B(x_{0},2r)$ and for all $\lambda \geq \eta ^{-1}$ and $0<\eta <1$%
, where $\rho =\eta r$.

We distinguish two cases depending on $J\neq 0$ or $J\equiv 0$.

\emph{Case }$J\neq 0$\emph{.} By (\ref{tt}), (\ref{44}) and using $\rho
=\eta r$, we have%
\begin{eqnarray}
q_{t}^{(\rho )}(x,y) &\leq &\frac{C_{8}}{t^{\alpha /\beta }}\exp (\,\frac{1}{%
4}R_{0}^{-\beta }t\,)\exp \left( 2C_{0}\lambda ^{2\beta +2}Tt-\lambda \right)
\notag \\
&=&\frac{C_{8}}{t^{\alpha /\beta }}\exp (\,\frac{1}{4}R_{0}^{-\beta
}t\,)\exp \left( 2C_{0}\lambda ^{2\beta +2}e^{c_{1}(\eta )\lambda }\frac{t}{%
\rho ^{\beta }}-\lambda \right)  \notag \\
&\leq &\frac{C_{8}}{t^{\alpha /\beta }}\exp (\,\frac{1}{4}R_{0}^{-\beta
}t\,)\exp \left( C_{9}(\eta )e^{2c_{1}(\eta )\lambda }\frac{t}{r^{\beta }}%
-\lambda \right)  \label{eq:q_est}
\end{eqnarray}%
where $c_{1}(\eta )=2(\beta +1)\left( \eta +2\eta ^{2}\right) $\ by (\ref%
{eq:c1_Def}), and $C_{9}(\eta )$ is given by%
\begin{equation*}
C_{9}(\eta )=2C_{0}\eta ^{-\beta }\left\{ 2\left( \beta +1\right)
/c_{1}(\eta )\right\} ^{2\beta +2}=2C_{0}\eta ^{-\beta }\left( \eta +2\eta
^{2}\right) ^{-2(\beta +1)},
\end{equation*}%
where in the last inequality we have used the following:%
\begin{equation*}
\lambda ^{2\beta +2}\leq \left\{ 2\left( \beta +1\right) /c_{1}(\eta
)\right\} ^{2\beta +2}e^{c_{1}(\eta )\lambda }
\end{equation*}%
by the elementary inequality $a\leq e^{a}$ for any $a\geq 0$, with $a=\frac{%
c_{1}(\eta )}{2(\beta +1)}\lambda =\left( \eta +2\eta ^{2}\right) \lambda $.

We first choose $\lambda $ and then choose $\eta $. Choose $\lambda $ such
that $e^{-\lambda }=\left( \frac{r}{t^{1/\beta }}\right) ^{-(\alpha +\beta
)} $, that is,%
\begin{equation}
\qquad \lambda =\frac{\alpha +\beta }{\beta }\log \left( r^{\beta }/t\right)
,  \label{eq:lam_choice}
\end{equation}%
but we need to ensure the condition $\lambda \geq \eta ^{-1}$ is satisfied,
namely%
\begin{equation}
\log \left( r^{\beta }/t\right) \geq \frac{\beta }{\alpha +\beta }\eta ^{-1}.
\label{13}
\end{equation}%
With such choice of $\lambda $, we then choose $\eta \in (0,1)$ such that 
\begin{equation}
e^{2c_{1}(\eta )\lambda }\frac{t}{r^{\beta }}=1,  \label{14}
\end{equation}%
that is, 
\begin{equation*}
4(\beta +1)\left( \eta +2\eta ^{2}\right) =2c_{1}(\eta )=\frac{\beta }{%
\alpha +\beta }.
\end{equation*}%
(Clearly this can be achieved. Actually we have $\eta +2\eta ^{2}\leq \frac{1%
}{4}$, implying $0<\eta <\frac{\sqrt{3}-1}{2}$). Once $\eta $ is chosen by (%
\ref{14}), then the condition (\ref{13}) is satisfied if 
\begin{equation}
r^{\beta }/t\geq c_{2}  \label{15}
\end{equation}%
for some universal constant $c_{2}>0$.

Therefore, we conclude from (\ref{eq:q_est}), (\ref{14}), (\ref%
{eq:lam_choice}) that 
\begin{equation}
q_{t}^{(\rho )}(x,y)\leq \frac{C_{8}}{t^{\alpha /\beta }}\exp (\,\frac{1}{4}%
R_{0}^{-\beta }t\,)\exp \left( C_{9}(\eta )\right) \cdot e^{-\lambda
}=C_{10}\exp (\,\frac{1}{4}R_{0}^{-\beta }t\,)\frac{t}{r^{\alpha +\beta }}
\label{qE0}
\end{equation}%
for all $t>0$, $r\in (0,R_{0}/2)$ with $r^{\beta }\geq c_{2}t$ and all $\rho
=\eta r$, for $\mu $-almost all $x\in B(x_{0},r),y\in B(x_{0},2r)^{c}$,
where $C_{10}$ is a universal constant independent of $x_{0},t,r,x,y$ and $%
R_{0}$.

Note that 
\begin{equation*}
p_{t}(x,y)\leq q_{t}^{(\rho )}(x,y)+2t\underset{{x\in M,\;y\in B(x,\rho )^{c}%
}}{\text{sup}}J(x,y),
\end{equation*}%
see \cite[Lemma 3.1 (c)]{bagrku2009hkub}, or \cite[(4.13) p.6412]%
{grhulau2014nonlocal}. It follows that, using $\rho =\eta r,$ 
\begin{equation*}
p_{t}(x,y)\leq C_{10}\exp (\,\frac{1}{4}R_{0}^{-\beta }t\,)\frac{t}{%
r^{\alpha +\beta }}+C\frac{2t}{\rho ^{\alpha +\beta }}\leq C_{11}\exp (\,%
\frac{1}{4}R_{0}^{-\beta }t\,)\frac{t}{r^{\alpha +\beta }}
\end{equation*}%
for all $t>0$, $r\in (0,R_{0}/2)$ with $r^{\beta }\geq c_{2}t$, for $\mu $%
-almost all $x\in B(x_{0},r),y\in B(x_{0},2r)^{c}$, where $C_{11}$ is a
universal constant independent of $x_{0},t,r,x,y$ and $R_{0}$.

With a certain amount of effort, we can say that%
\begin{equation}
p_{t}(x,y)\leq C_{12}\exp (\,\frac{1}{4}R_{0}^{-\beta }t\,)\frac{t}{%
d(x,y)^{\alpha +\beta }}  \label{Ex}
\end{equation}%
for all $t>0$ and $\mu $-almost all $x,y\in M$, if $d(x,y)\geq
c_{3}t^{1/\beta }$, for some universal constants $C_{12}>0$ and $c_{3}>0$,
both of which are independent of $R_{0}$, thus showing that (\ref{eq:UE}) is
true.

Finally, if $d(x,y)<c_{3}t^{1/\beta }$ then (\ref{eq:UE}) follows directly
from (\ref{eq:DUE}).

\emph{Case }$J\equiv 0$\emph{.} By (\ref{tt}), (\ref{44}) and setting $\rho =%
\frac{1}{2}r$ with $\eta =\frac{1}{2}$, we have 
\begin{align*}
p_{t}(x,y)& =q_{t}^{(\rho )}(x,y)\leq \frac{C_{8}}{t^{\alpha /\beta }}\exp
(\,\frac{1}{4}R_{0}^{-\beta }t\,)\exp \left( 2C_{0}\lambda ^{2\beta
+2}Tt-\lambda \right) \\
& =\frac{C_{8}}{t^{\alpha /\beta }}\exp (\,\frac{1}{4}R_{0}^{-\beta
}t\,)\exp \left( 2C_{0}\lambda ^{2\beta +2}\frac{t}{r^{\beta }}-\lambda
\right)
\end{align*}%
for all $t>0,r\in \left( 0,R_{0}/2\right) $ and $\mu $-almost all $x\in
B(x_{0},r),y\in B(x_{0},2r)^{c}$, for all $\lambda \geq \eta ^{-1}=2$.
Choosing $\lambda $ such that%
\begin{equation*}
2C_{0}\lambda ^{2\beta +2}\frac{t}{r^{\beta }}=\frac{\lambda }{2},
\end{equation*}%
that is, $\lambda =\left( \frac{1}{4C_{0}}\frac{r^{\beta }}{t}\right)
^{1/(2\beta +1)}$. But we need ensure that $\lambda \geq \eta ^{-1}=2$; this
can be achieved if $r^{\beta }\geq c_{4}t$ for some universal constant $%
c_{4}>0$. Therefore, we obtain 
\begin{equation*}
p_{t}(x,y)\leq \frac{C_{8}}{t^{\alpha /\beta }}\exp (\,\frac{1}{4}%
R_{0}^{-\beta }t\,)\exp \left( -c\left( \frac{r^{\beta }}{t}\right)
^{1/(2\beta +1)}\right)
\end{equation*}%
for all $t>0$, $r\in (0,R_{0}/2)$ with $r^{\beta }\geq c_{4}t$ and $\mu $%
-almost all $x\in B(x_{0},r),y\in B(x_{0},2r)^{c}$, where $C_{8},c$ are
independent of $R_{0}$.

Therefore, we conclude that%
\begin{equation}
p_{t}(x,y)\leq \frac{C}{t^{\alpha /\beta }}\exp (\,\frac{1}{4}R_{0}^{-\beta
}t\,)\exp \left( -c\left( \frac{d(x,y)}{t^{1/\beta }}\right) ^{\frac{\beta }{%
\beta ^{\prime }-1}}\right)  \label{C12}
\end{equation}%
for all $t>0$ and $\mu $-almost all $x,y\in M$, where 
\begin{equation*}
\beta ^{\prime }:=2\beta +2.
\end{equation*}

Finally, we obtain \ref{eq:UE_loc} by applying Lemma \ref{UEiter}. The proof
is complete.
\end{proof}

We finish this section by proving Theorem \ref{T2}.

\begin{proof}[Proof of Theorem \protect\ref{T2}]
Indeed, by Theorem \ref{thm:Mr2}, it suffices to show the following opposite
implications 
\begin{eqnarray}
(\ref{eq:UE}) &\Rightarrow &(\ref{eq:DUE})+(\ref{eq:CSA})+\ref{eq:J_less},
\label{33} \\
\ref{eq:UE_loc} &\Rightarrow &(\ref{eq:DUE})+(\ref{eq:CSA})+(J\equiv 0).
\label{37}
\end{eqnarray}%
Indeed, it is trivial to see that (\ref{eq:DUE}) follows either from (\ref%
{eq:UE}) or from \ref{eq:UE_loc}, whilst the implication%
\begin{equation*}
(\ref{eq:UE})\Rightarrow \ref{eq:J_less}
\end{equation*}%
was proved in \cite[p.150]{bagrku2009hkub}. On the other hand, the
implication 
\begin{equation*}
\ref{eq:UE_loc}\Rightarrow (J\equiv 0)
\end{equation*}%
follows by using the fact that%
\begin{equation*}
J(x,y)=\lim_{t\rightarrow 0}\frac{1}{2t}p_{t}(x,y)\text{ for }\mu \text{%
-a.a. }(x,y)\in M\times M\setminus \text{diag}.
\end{equation*}%
(alternatively $J\equiv 0$ follows from \cite[Theorem 3.4]{grku2008dic} no
matter $R_{0}<\infty $ or $R_{0}=\infty $).

Therefore, the implications (\ref{33}), (\ref{37}) will follow if we can show%
\begin{eqnarray}
(\ref{eq:UE}) &\Rightarrow &(\ref{eq:S}),  \label{41} \\
\ref{eq:UE_loc} &\Rightarrow &(\ref{eq:S}),  \label{42}
\end{eqnarray}%
since we have already proved 
\begin{equation*}
(\ref{eq:S})\Rightarrow (\ref{eq:CSA})
\end{equation*}%
in Lemma \ref{lem:S_to_CSA} in Section \ref{sec:CSA}.

To prove (\ref{41}), (\ref{42}), let $B:=B(x_{0},r)$ for $x_{0}\in M$ and $%
r\in (0,R_{0}).$ Note that if the heat kernel $p_{t}(x,y)$ satisfies 
\begin{equation}
p_{t}(x,y)\leq \frac{C}{t^{\alpha /\beta }}\exp \left( R_{0}^{-\beta
}t\right) \Phi _{2}\left( \frac{d(x,y)}{t^{1/\beta }}\right)  \label{38}
\end{equation}%
for all $t>0$ and $\mu $-almost all $x,y\in M$, where $\Phi _{2}$ is some
non-increasing function on $[0,\infty )$, then using condition ($V_{\leq }$%
), 
\begin{equation}
P_{t}1_{B^{c}}(x)\leq C\int_{\frac{1}{8}rt^{-1/\beta }}^{\infty }s^{\alpha
-1}\Phi _{2}(s)ds  \label{40}
\end{equation}%
for all $x\in \frac{1}{2}B$, for some constant $C$ independent of $%
x_{0},r,R_{0}$ (see \cite[formula (3.7)]{grigor2003heat}). Assume further
that%
\begin{equation}
\int_{0}^{\infty }s^{\alpha -1}\Phi _{2}(s)ds<\infty .  \label{39}
\end{equation}%
Then by (\ref{40}), (\ref{39}),%
\begin{equation*}
P_{t}1_{B^{c}}\leq \frac{1}{2}\text{ in }\frac{1}{2}B
\end{equation*}%
if $rt^{-1/\beta }\gg 1$. From this and using the conservativeness of $(%
\mathcal{E},\mathcal{F})$, we obtain condition ($S$) (see \cite[Theorem 5.8
p.544, and Remark 5.9 p.547]{grigor2014upper}). Since the assumptions (\ref%
{38}) and (\ref{39}) are satisfied either by condition (\ref{eq:UE}) where%
\begin{equation*}
\Phi _{2}(s)=(1+s)^{-(\alpha +\beta )}
\end{equation*}%
or by \ref{eq:UE_loc} where%
\begin{equation*}
\Phi _{2}(s)=\exp \left( -cs^{\beta /(\beta -1)}\right)
\end{equation*}%
for all $s\geq 0$, we conclude that (\ref{41}), (\ref{42}) hold. The proof
is complete.
\end{proof}

\bibliographystyle{siam}
\bibliography{Ref_UHK}

\end{document}

%% file: tcirm.tex
\def\RM{\rm}

\def\oint{\fint}

\def\FRAME#1#2#3#4#5#6#7#8
{
 \begin{figure}[ptbh]
 \begin{center}
 \includegraphics[height=#3]{#7}
 \caption{#5}
 \label{#6}
 \end{center}
 \end{figure}
}